\documentclass[9pt]{jdg-p-1-2}

\newtheorem{theorem}{Theorem}[section]
\newtheorem{lemma}[theorem]{Lemma}
\newtheorem{corollary}[theorem]{Corollary}
\newtheorem{proposition}[theorem]{Proposition}

\theoremstyle{definition}
\newtheorem{definition}[theorem]{Definition}

\theoremstyle{remark}
\newtheorem{remark}[theorem]{Remark}

\def\Ric{\text{Ric}}

\def\e{\epsilon}

\def\R{\Bbb R}

\def\id{\operatorname{id}}
\def\Ric{\operatorname{Ric}}

\def\tr{\operatorname{tr}}

\newcommand{\pd}[2]{\frac{\partial #1}{\partial #2}}
\newcommand{\pdt}{\pd{}{t}}

\numberwithin{equation}{section}

\begin{document}

\title[Gradient estimates and an entropy formula]{Local gradient estimates of $p$-harmonic functions, $1/H$-flow, and an entropy formula}



\author{Brett Kotschwar}
\address{Department of Mathematics, MIT, Cambridge, MA 02139}

\email{kotschwar@math.mit.edu}

\author{Lei Ni}
\address{Department of Mathematics, University of California at San Diego, La Jolla, CA 92093}

\email{lni@math.ucsd.edu}

\thanks{The first author was supported in part by NSF grant DMS-034540. The second author was
supported in part by NSF grant DMS-0504792  and an Alfred P. Sloan
Fellowship, USA}



\date{July 2007}

\begin{abstract}
  In the first part of this paper, we prove local 
interior and boundary gradient estimates for $p$-harmonic
functions on general Riemannian manifolds.  With these estimates, following
the strategy in recent work of  R. Moser,
we prove an existence theorem for weak solutions to the level set formulation of 
the $1/H$ (inverse mean curvature)
flow for hypersurfaces in ambient manifolds satisfying a sharp volume growth assumption.
In the second part of this paper, we consider two parabolic analogues of the $p$-Laplace equation
and prove sharp Li-Yau type gradient estimates for positive solutions to these equations on
manifolds of nonnegative Ricci curvature.  For one of these equations, we also prove an 
entropy monotonicity formula generalizing an earlier such formula of the 
second author for the linear heat equation. As an application of this formula, we show 
that a complete Riemannian
manifold with non-negative Ricci curvature and sharp $L^p$-logarithmic Sobolev inequality
must be isometric to Euclidean space.
\end{abstract}
\maketitle

\section{Introduction}

Recently, in \cite{Moser-j}, an interesting connection between the
$p$-harmonic functions and the $1/H$ hypersurface flow (also called
the inverse mean curvature flow) was established.  Let $v$ be a positive
$p$-harmonic function, i.e., a function satisfying
\begin{equation}\label{eq-v}
\operatorname{div} \left(|\nabla v|^{p-2}\nabla v\right) =0,
\end{equation}
and let $u\doteqdot-(p-1)\log v$. It is easy to see that $u$ satisfies
\begin{equation}\label{eq-u}
\operatorname{div}\left(|\nabla u|^{p-2}\nabla u\right)=|\nabla
u|^p.
\end{equation}
The objective in \cite{Moser-j} is to obtain a weak solution to the
$1/H$ flow in the level-set formulation of \cite{IH} 
on a co-compact subdomain $\Omega$ of $\R^n$. Where it is sufficiently regular,
such a solution $u$
will satisfy
\begin{equation}\label{1/h}
\operatorname{div}\left(\frac{\nabla u}{|\nabla u|}\right)=|\nabla u|,
\end{equation}
that is, \eqref{eq-u} with $p=1$. Moser's strategy in \cite{Moser-j} is to obtain
weak solutions $u$ to \eqref{1/h} as limits of solutions $u^{(p)}$ to \eqref{eq-u}
as $p\searrow 1$. To obtain these solutions $u^{(p)}$,  he first uses explicit barriers
to solve \eqref{eq-v} with appropriate Dirichlet boundary conditions and applies the correspondence above.  The key ingredient of the convergence, and to the success of this strategy, is to obtain a gradient estimate on the solutions
$u = u^{(p)}$.  The estimate is aided by the following observation. If
$$
f\doteqdot |\nabla u|^2\doteqdot (p-1)^2\frac{|\nabla v|^2}{v^2},
$$
then, expressed in terms of $f$, (\ref{eq-u}) has the
equivalent form
\begin{equation}\label{eq-u2}
(\frac{p}{2}-1)f^{p/2-2}\langle \nabla f, \nabla u\rangle
+f^{p/2-1}\Delta u =f^{p/2}.
\end{equation}
Using this equation, the gradient estimate in
\cite{Moser-j} is proven from out of a boundary estimate, which in turn is derived
from a $C^0$ estimate by way of explicit barriers, certain integral estimates on
$|\nabla v|$, and a Harnack inequality for $p$-harmonic functions on Euclidean spaces.
The construction of the barriers in particular rely on the underlying manifold being Euclidean.
The regularization procedure and also the reduction of the convergence argument
to a certain gradient estimate, uniform in $p$ as $p\to 1$,  have some precedence, for example,
in \cite{EI}, in other contexts.

In the first part of this paper, we
derive interior and boundary gradient estimates on a general
Riemannian manifold $(M, g)$ via the gradient estimate technique of
\cite{Ch-Yau, LY} and use it to establish an existence result for
the $1/H$-flow on a class of complete Riemannian  manifolds. A new feature of our derivation
of the local estimate is a nonlinear Bochner type formula relating
the nonlinear operator with its linearization.

We first start with an interior/local  estimate for positive
$p$-harmonic functions, which may be of independent  interest.

\begin{theorem}\label{main1}
Assume that $v$ is a positive $p$-harmonic function on the ball
$B(x_0, R)$, and that on
 the ball $B(x_0, R)$ the sectional curvature of $(M, g)$, $K_M \ge
-K^2$. Then for any $\epsilon>0$,
\begin{equation}\label{local-gre}
\sup_{B(x_0, \frac{R}{2})}|\nabla u|^2 \le
\frac{5(n-1)}{R^2(1-\epsilon)}\left(c_{p, n}+\frac{(n-1)b_{p,
n}^2}{8\epsilon}\right)+C(n, K, p, R, \epsilon)
\end{equation}
where \begin{eqnarray*} b_{p,
n}&=&\left(\frac{2(p-1)}{n-1}-2\right),\\
 c_{p,
n}&=&\left(\frac{2(p/2-1)^2}{n-1}+\frac{(p-2)p}{2}\right)_{-}+\left((\frac{p}{2}+1)\max
\{p-1, 1\}\right),
\\
C(n, K, p, R,
\epsilon)&=&\frac{(n-1)^2}{1-\epsilon}K^2+\frac{10(n+p-2)(n-1)}{1-\epsilon}\frac{1+KR}{R^2}
\\
&\quad &+\frac{5\max \{p-1, 1\}(n-1)}{(1-\epsilon)R^2}.
\end{eqnarray*}
\end{theorem}

Note that $b_{p, n}$, $c_{p, n}$ and $C(n, K, p, R, \epsilon)$ all
stay finite as $p\to 1$. Hence Theorem \ref{main1} effectively gives
an estimate for the gradient of the solution to $1/H$-flow. Also, if
$v$ is defined globally, by taking $R\to \infty$, and $\epsilon \to
0$, Theorem \ref{main1}  implies that for  any positive $p$-harmonic
function $v$, $u=-(p-1)\log v$ satisfies
\begin{equation}\label{global-gre}
|\nabla u|^2 \le (n-1)^2K^2.
\end{equation}
The constant $(n-1)^2$ is sharp in light of the results of
\cite{Yau1, LW} for the case $p=2$. From the proof it is evident that when $p=2$, one
can relax the assumption $K_M \ge -K^2$ by $\Ric\ge -(n-1)K^2$. A
direct consequence of this is that if $(M, g)$ is a complete
manifold with nonnegative sectional curvature, then any positive
$p$-harmonic function must be a constant. In fact, using the gradient
technique of this paper, one can prove that any positive $p$-harmonic function
must be a constant if $M$ has sectional curvature bounded from below
and its Ricci curvature is nonnegative. However, this result can be
obtained by only assuming that the manifold has the so-called volume
doubling property and satisfies a Poincar\'e type inequality (see,
for example, \cite{Ho} as well as Section 4 for more details). Hence
it holds in particular on any Riemannian manifold with nonnegative
Ricci curvature. On the other hand, it is not clear whether one can
obtain an estimate on $|\nabla u|^2$ such as (\ref{global-gre}) under the
weaker assumption of a Ricci curvature lower bound.

 With the help of the interior gradient estimate, by constructing suitable
  (local) barrier functions
 we can establish the following boundary estimate:

\medskip
 {\it For every $\epsilon > 0$, there exists $p_0 =
p(\epsilon)
> 1$ such that if $u$ satisfies equation \eqref{eq-u} on $\Omega$
for some $1 < p \leq p_0$, then we have the estimate
\begin{equation}
\label{boundarymc}
    \left|\nabla u\right| \leq H_{+} + \epsilon
\end{equation}
where $H$ denotes the mean curvature of $\partial\Omega$ and
$H_{+}(x) = \max\left\{ H(x), 0\right\}$. }

\medskip

A similar boundary estimate was first proved in \cite{IH} for
solutions to a different equation approximating \eqref{1/h} (called elliptic
regularization). Our method is a modification of theirs.

With the help of the interior and boundary estimates above,
and following the general scheme of \cite{Moser-j}, one can prove the existence of a {\it proper solution} (please see
Section 4 for the definition) for a class of Riemannian
manifolds which include the asymptotically Euclidean manifolds considered in
\cite{IH}. The following is a special case implied by our general
existence  theorem.

\begin{theorem} Let $M$ be a complete Riemannian manifold such that
its sectional curvature $K_M(x)\ge -k(r(x))$ for some nonincreasing
function $k(t)$ with $\int_0^\infty t k(t)\, dt < \infty$. Let
$\Omega$ be an end of $M$. Assume that for some $p_0>1$,
$$\int_1^\infty \left(\frac{t}{V(\Omega\cap B(o,
t))}\right)^{1/(p_0-1)}\, dt < \infty$$ and
$$
\lim_{r\to \infty}\sup_{2r\le t< \infty}\frac{t}{V(\Omega \cap B(o,
t ))}=0.
$$
Then (\ref{1/h}) has a proper weak solution $u$ with $\lim_{x\to
\infty}|\nabla u|(x)=0$.
\end{theorem}
The volume growth conditions in the theorem are optimal for the
existence of {\it proper} solutions. This is  shown  in
Section 4.

After the uniform gradient estimate, the key of the proof of the
existence theorem is to construct certain bounded $p$-harmonic
functions and obtain effective $C^0$-estimates of such $p$-harmonic functions at infinity
which hold up as $p\to 1$. Since we need
to ensure that the limit as $p\to 1$ is a nonconstant function which
goes to $+\infty$ at the infinity of manifold, the estimate is
somewhat delicate. Here we rely crucially on the earlier work of
Holopainen \cite{Ho}. One may expect more general existence results,
for example, for manifolds with a Laplacian whose
$L^2$-spectrum has a positive lower bound \cite{LW}. However, it
seems that further work, with perhaps more delicate estimates, may be needed.

In the second part of the paper, which starts from Section 5, we consider some nonlinear parabolic
equations motivated by Theorem \ref{main1}. First, we prove sharp
gradient estimates of Li-Yau type for two nonlinear parabolic
equations associated with (\ref{eq-v}). Both estimates proved are
sharp in the case that $M$ has nonnegative Ricci curvature and
provide nonlinear generalizations of Li-Yau's estimate for the heat
equation. One of these estimates was obtained earlier in \cite{EV1,
EV2} (see also \cite{V1}) for the special case $M=\R^n$, in the
study of the regularity of nonnegative weak solutions.  We refer the
readers to Section 5 and 6 for the more detailed discussions on
these results.

A little surprisingly, we  also obtain an  entropy formula for a
class of nonlinear parabolic equations,  generalizing the earlier formula
for the linear heat equation in \cite{N1}. More precisely, we show the result below.

\begin{theorem} Let $(M, g)$ be a complete Riemannian manifold.  For any $p>1$,
let $v$ be a positive solution to the equation
\begin{equation}\label{ddnonl}
\frac{\partial v^{p-1}}{\partial t}=(p-1)^{p-1}
\operatorname{div}(|\nabla v|^{p-2}\nabla v)
\end{equation}
satisfying $\int v^{p-1}d\mu =1$. Then $$\frac{d}{d t}\mathcal{W}_p(v,
t)=-tp\int_M \left(\left|f^{p/2-1}\nabla_i \nabla_j u
-\frac{1}{tp}a_{ij}\right|^2_{A}+f^{p-2}R_{ij} u_i u_j\right)
v^{p-1}\, d\mu
$$
with $u=-(p-1)\log v$, $f=|\nabla u|^2$,
 $a_{ij}=g_{ij}-\frac{p-2}{p-1}\frac{v_i v_j}{|\nabla v|^2}$, $|T|_A^2=A^{ik}A^{jl}T_{ij}T_{kl}$
  for any $2$-tensor $T$ where $(A^{ij})$ is the inverse of $(a_{ij})$. The entropy
  $$\mathcal{W}_p(v, t)=\int_M
(t|\nabla \varphi|^p+\varphi -n)v^{p-1}\, d\mu
$$
is defined  with
$v^{p-1}=\frac{1}{\pi^{n/2}({p^*}^{p-1}p)^{\frac{n}{p}}}\frac{\Gamma(n/2+1)}
{\Gamma(n/p^*+1)}\frac{e^{-\varphi}}{t^{\frac{n}{p}}}$,
 where $p^*=\frac{p}{p-1}$, and is assumed to be finite.
\end{theorem}

The nonlinear heat equation (\ref{ddnonl}) has been the object of some previous
study. See, for example, \cite{V1} for a nice survey on the subject.
When $p=1$, the above theorem limits to the following form.

\begin{theorem} \label{efp=1} Let $(M, g)$ be a complete Riemannian manifold.
Let $u$ be a solution to
\begin{equation}\label{1/hpar}
\frac{\partial u}{\partial t} -\operatorname{div}(|\nabla
u|^{-1}\nabla u)+|\nabla u|=0
\end{equation}
satisfying that $\int_M e^{-u}\, d\mu =1$.  Then
$$\frac{d}{d
t}\mathcal{W}_1(u, t)=-t\int_M \left(\left|f^{-1/2}\nabla_\alpha
\nabla_\beta u -\frac{1}{t}g_{\alpha\beta}\right|^2+
\frac{1}{t^2}+f^{-1}R_{ij} u_i u_j\right) e^{-u}\, d\mu
$$
with $f=|\nabla u|^2$, $2\le \alpha, \beta \le n$ are with respect
to the orthonormal frame $e_1=\frac{\nabla u}{|\nabla u|}, e_2,
\cdot\cdot\cdot, e_n$. The entropy
$$
\mathcal{W}_1(u, t)=\int_M \left( t|\nabla u|+u-n\log t -\log
\omega_n-n\right)e^{-u}\, d\mu,
$$
where $\omega_n$ is the volume of unit ball in $\R^n$, is assumed to
be finite.
\end{theorem}
Note that (\ref{1/hpar}) is the parabolic equation associated to
(\ref{eq-u}), the level-set formulation of the $1/H$-flow equation; it has
been studied in \cite{Hein}.  In view of the importance of the
entropy formula of Perelman \cite{P},
 we expect that the above result will also play a role in understanding
the  analytical properties of this nonlinear parabolic equation. In
fact, as a simple consequence of the entropy formula, one can
conclude that on manifolds with nonnegative Ricci curvature, {\it any
positive  ancient solution to (\ref{ddnonl}) must be
constant}. Similar to the case of $p=2$, the entropy $\mathcal{W}(\varphi, t)$
is closely related to the optimal $L^p$ logarithmic Sobolev
inequality (cf. \cite{DD, G}) of Euclidean space. When $M=\R^n$, the
optimal $L^p$ Sobolev logarithmic Sobolev inequality implies that
$\mathcal{W}(\varphi, t)\ge 0$ for any $v^{p-1}$ with $\int
v^{p-1}\, d\mu=1$. As another application, we prove the following statement:

\medskip

{\it Let $(M, g)$ be a complete Riemannian manifold with nonnegative
Ricci curvature. Assume that the $L^p$-logarithmic Sobolev
inequality  (\ref{log-sob1}) holds with the sharp constant on $M$
for some $p>1$. Then $M$ is isometric to $\R^n$.}

\medskip

This generalizes the $p=2$ case which was originally proved in
\cite{BCL} (see also \cite{N1} for a different proof via the entropy
formula for the linear heat equation).

The entropy formula  coupled with Ricci flow is of special interest.
This will be the subject of a forth-coming paper.

Finally we study the localized version of the sharp gradient
estimates of Li-Yau type mentioned previously for manifolds with
lower bound on the sectional curvature. This is carried out  in the
last section.

\bigskip
 {\it Acknowledgement}.  We would like to thank B. Chow for conservations,
 K Ecker and T. Ilmanen
 for their interests and  helpful discussions.

\section{The proof of Theorem \ref{main1}}
\label{sec:mainproof} Let $u$, $f$ be as above. Assume that $f>0$
over some region of $M$. As in \cite{Moser-j}, define
$$
\mathcal{L}(\psi) \doteqdot \operatorname{div} \left( f^{p/2-1}A
(\nabla \psi)\right)-pf^{p/2-1}\langle \nabla u, \nabla \psi\rangle.
$$
Here
$$
A=\id +(p-2)\frac{\nabla u\otimes \nabla u}{f}
$$
which can be checked easily to be nonnegative definite in general
and positive definite for $p>1$. Note that the operator
$\mathcal{L}$ is the linearized operator of the nonlinear equation
(\ref{eq-u}).

 The first is a computational lemma.
\begin{lemma}\label{bochner1}
\begin{equation}\label{eq-help1}
\mathcal{L}(f)=2f^{p/2-1}\left(u_{ij}^2+R_{ij} u_i
u_j\right)+\left(\frac{p}{2}-1\right)|\nabla f|^2 f^{p/2-2}.
\end{equation}
Here $u_{ij}$ is the Hessian of $u$. $R_{ij}$ is the Ricci curvature
of $M$.
\end{lemma}
\begin{proof} Direct calculation shows that
\begin{eqnarray*}
\mathcal{L} (f)&=& (\frac{p}{2}-1)f^{p/2-2}|\nabla f|^2
+f^{p/2-1}\Delta f +(p-2)\Delta u \langle \nabla u, \nabla f\rangle
f^{p/2-2}\\
&\quad& +(p-2)(\frac{p}{2}-1)f^{p/2-3} \langle \nabla u, \nabla
f\rangle^2 \\
&\quad& +(p-2)\left(u_{ij} f_i u_j f^{p/2-2} +f_{ij} u_i u_j
f^{p/2-2}-\langle \nabla u, \nabla f \rangle ^2
f^{p/2-3}\right)\\
&\quad & -pf^{p/2-1}\langle \nabla u, \nabla f\rangle.
\end{eqnarray*}
Using
$$
\Delta f =2u_{ij}^2 +2\langle \nabla \Delta u, \nabla u\rangle
+2R_{ij}u_i u_j
$$
and combining terms we have that
\begin{eqnarray*}
\mathcal{L} (f)&=& f^{p/2-1}\left(2u_{ij}^2+2\langle\nabla \Delta u,
\nabla u\rangle +2R_{ij}u_i u_j\right)+
(\frac{p}{2}-1)f^{p/2-2}|\nabla
f|^2 \\
&\quad &+(p-2)\Delta u \langle \nabla u, \nabla f\rangle f^{p/2-2}
+(p-2)(\frac{p}{2}-2)f^{p/2-3} \langle \nabla u, \nabla
f\rangle^2 \\
&\quad& +(p-2)\left(u_{ij} f_i u_j f^{p/2-2} +f_{ij} u_i u_j
f^{p/2-2}\right) -pf^{p/2-1}\langle \nabla u, \nabla f\rangle.
\end{eqnarray*}
Taking the gradient of both sides of (\ref{eq-u2}) and computing its product with
$\nabla u$, we have that
\begin{eqnarray*}
&\, &(\frac{p}{2}-1)(\frac{p}{2}-2)f^{p/2-3} \langle \nabla f,
\nabla u \rangle ^2 +(\frac{p}{2}-1)f^{p/2-2}\left(f_{ij}u_i u_j
+u_{ij}f_i u_j\right)\\
&\quad & +(\frac{p}{2}-1)f^{p/2-2}\Delta u \langle \nabla f, \nabla
u\rangle +f^{p/2-1}\langle \nabla \Delta u, \nabla u\rangle  =
\frac{p}{2}f^{p/2-1}\langle \nabla f, \nabla u\rangle.
\end{eqnarray*}
Combining the above two, we prove the claimed identity
(\ref{eq-help1}).
\end{proof}

Now let $\eta(x)=\theta \left(\frac{r(x)}{R}\right)$, where $\theta
(t)$ is a cut-off function such that $\theta (t)\equiv 1$ for $0\le
t \le \frac{1}{2}$ and $\theta(t)\equiv 0$ for $t\ge 1$. Furthermore, take
the derivatives of $\theta$ to satisfy
$\frac{(\theta')^2}{\theta} \le 10$ and $\theta''\ge -10\theta\ge
-10$. Here $r(x)$ denotes the distance from some fixed $x_0$. Let
$Q=\eta f$, which vanishes outside $B(x_0, R)$. At the maximum point
of $Q$, it is easy to see that
\begin{equation}\label{eq-gr0}
\nabla Q =(\nabla \eta)f +(\nabla f)\eta =0
\end{equation}
and
$$
0\ge \mathcal{L}(Q).
$$
On the other hand, at the maximum point,
\begin{eqnarray*}
\mathcal{L}(Q) &=& \eta\operatorname{div} \left(f^{p/2-1} A \nabla
f\right)-\eta p f^{p/2-1}\langle \nabla u, \nabla f\rangle \\
&\, & +f^{p/2-1}\langle A(\nabla f), \nabla \eta\rangle
+\operatorname{div} (f^{p/2} A \nabla \eta)-pf^{p/2}\langle \nabla
u, \nabla \eta \rangle\\
&=&\eta \mathcal{L}(f) -(\frac{p}{2}+1)f^{p/2}\frac{\langle A(\nabla
\eta), \nabla \eta \rangle}{\eta} +f^{p/2}\operatorname{div}(A\nabla
\eta)-pf^{\frac{p}{2}}\langle \nabla u, \nabla \eta\rangle.
\end{eqnarray*}
If  $1<p\le 2$,
$$
\frac{\langle A(\nabla \eta), \nabla \eta \rangle}{\eta}\le
\frac{|\nabla \eta|^2}{\eta}.
$$
For $p\ge2$ case we have that
$$
\frac{\langle A(\nabla \eta), \nabla \eta \rangle}{\eta}\le
(p-1)\frac{|\nabla \eta|^2}{\eta}.
$$
The next lemma estimates $\operatorname{div}(A\nabla \eta)$.

\begin{lemma}\label{cf-est1} Assume that, on $B(x_0, R)$, the sectional curvature of
$(M, g)$ satisfies $K_M \ge - K^2$. Then, at the maximum point
of $Q$,
\begin{eqnarray}\label{ieq-div}
\operatorname{div}(A\nabla \eta)&\ge& -20(n+p-2)\frac{1+KR}{R^2}
-\frac{10\max \{p-1, 1\}}{R^2}+(p-2)\langle \nabla u, \nabla
\eta\rangle  \nonumber\\
&\quad&+(p-2)\frac{p}{2}\frac{\langle \nabla u, \nabla
\eta\rangle^2}{\eta f}-(\frac{p}{2}-1)\frac{|\nabla \eta|^2}{\eta}.
\end{eqnarray}
\end{lemma}
\begin{proof} Direct computation yields
\begin{eqnarray*}
\operatorname{div}(A\nabla \eta)&=& \Delta \eta
+(p-2)\frac{\eta_{ij} u_i u_j}{f} +(p-2)\Delta u \frac{\langle
\nabla u, \nabla \eta\rangle}{f}\\
&\quad&-(p-2)\frac{\langle \nabla u, \nabla f\rangle \langle \nabla
u, \nabla \eta\rangle}{f^2} + (p-2) \frac{u_{ij}u_i \eta_j}{f}.
\end{eqnarray*}
Now using (\ref{eq-gr0}), (\ref{eq-u2}) and that $ f_j =2u_{ij}u_i$,
we can eliminate $\Delta u$ and $\nabla f$ to arrive at
\begin{eqnarray*}
\operatorname{div}(A\nabla \eta)&=& \Delta \eta
+(p-2)\frac{\eta_{ij} u_i u_j}{f} +(p-2)\langle \nabla u, \nabla \eta\rangle\\
&\quad&+(p-2)\frac{p}{2}\frac{\langle \nabla u, \nabla \eta\rangle^2
}{f\eta} -(\frac{p}{2}-1) \frac{|\nabla \eta|^2}{\eta}.
\end{eqnarray*}
We only need to estimate the first two terms, for which we compute
$$
\eta_{ij}=\theta'\frac{ r_{ij}}{R} +\theta'' \frac{r_i r_j}{R^2}.
$$
Using the Hessian comparison theorem \cite {GW}, which states that
$r_{ij}\le \frac{1+Kr}{r} g_{ij}$, and the Laplacian comparison
theorem, we have that
$$ A_{ij} r_{ij}
 \le (n+p-2)\left(\frac{1+Kr}{r} \right).
$$
Noting that $\theta'=0$ if $r\le \frac{R}{2}$,  we have that
\begin{eqnarray*}
\Delta \eta +(p-2)\frac{\eta_{ij} u_i u_j}{f} &=&A_{ij}\eta_{ij}
\\
&\ge& -20(n+p-2)\frac{1+KR}{R^2}-10\max \{ p-1, 1\}\frac{1}{R^2}.
\end{eqnarray*}
Taken together, these estimates prove the lemma.
\end{proof}

\begin{remark} The above lemma is the only place we need to assume
that the sectional curvature of $M$ is bounded from below by $-K^2$.
We expect that by some judicious choice of cut-off function, one
may be able to prove Theorem \ref{main1} only assuming that the
Ricci curvature is bounded from below.
\end{remark}

To prove the theorem, we first estimate $\mathcal{L}(f)$ from below.
We only need to estimate it over the points where $f>0$ for our
purpose of estimating $f$ from above. Choose a local orthonormal
frame $\{e_i\}$ near any such given point so that at the given point
$\nabla u=|\nabla u| e_1$. Then $f_1 =2 u_{j1}u_j=2u_{11} u_1$ and
for $j\ge 2$, $f_j=2u_{j1}u_1$, which implies that
\begin{equation}\label{help1}
2u_{k1}=\frac{f_k}{f^{1/2}}.
\end{equation}
Now (\ref{eq-u2}) becomes
$$
\sum_{j\ge 2}u_{jj}=f-(\frac{p}{2}-1)\frac{f_1 u_1}{f}-u_{11}.
$$
Hence
\begin{eqnarray*}
\sum_{i,j=1}^n u^2_{ij}&\ge& u^2_{11}+2\sum_{j\ge 2} u_{j1}^2
+\sum_{j\ge 2}
u_{jj}^2\\
&\ge & u^2_{11}+2\sum_{j\ge 2} u_{j1}^2 +\frac{1}{n-1}(\sum_{j\ge 2}
u_{jj})^2\\
&=& \frac{n}{n-1}u^2_{11}+2\sum_{j\ge 2} u_{j1}^2+\frac{1}{n-1}f^2
+\frac{1}{n-1}(\frac{p}{2}-1)^2\frac{(f_1 u_1)^2}{f^2}\\
&\,& -\frac{2}{n-1}(\frac{p}{2}-1)f_1 u_1 -\frac{2}{n-1} f u_{11}
+\frac{2}{n-1}(\frac{p}{2}-1)\frac{f_1 u_1 u_{11}}{f}.
\end{eqnarray*}
Using (\ref{help1}), we can replace all the second derivatives of $u$
and arrive at
\begin{eqnarray}\label{help2}
\sum_{i, j=1}^n u_{ij}^2 &\ge& \frac{1}{n-1} f^2
+\frac{1}{n-1}\left(\frac{n}{4}+\frac{p}{2}-1\right)\frac{f_1^2}{f}+\frac{1}{2}\sum_{j\ge
2} \frac{f_j^2}{f} \nonumber\\
&\, & +\frac{1}{n-1}(\frac{p}{2}-1)^2 \frac{\langle\nabla f, \nabla
u\rangle^2}{f^2}-\frac{p-1}{n-1}\langle \nabla f, \nabla u \rangle\nonumber\\
&\ge& \frac{1}{n-1} f^2+a_{n, p} \frac{|\nabla
f|^2}{f}+\frac{1}{n-1}(\frac{p}{2}-1)^2 \frac{\langle\nabla f,
\nabla u\rangle^2}{f^2}-\frac{p-1}{n-1}\langle \nabla f, \nabla u
\rangle
\end{eqnarray}
where
$$
a_{n,
p}\doteqdot\min\{\frac{1}{n-1}\left(\frac{n}{4}+\frac{p}{2}-1\right),
\frac{1}{2}\}\ge 0.
$$
Hence by (\ref{eq-help1}), (\ref{eq-gr0}) we have that
\begin{eqnarray}
f^{p/2-1}\mathcal{L}(f)&\ge& \frac{2}{n-1} f^p +2a_{n, p}
f^{p-3}|\nabla f|^2 +\frac{2(p/2-1)^2}{n-1}f^{p-2}\frac{\langle
\nabla \eta, \nabla
u\rangle^2}{\eta^2} \nonumber\\
&\quad&+\frac{2(p-1)}{n-1} f^{p-1}\frac{\langle \nabla \eta, \nabla
u\rangle}{\eta}-2(n-1)K^2f^{p-1}+(\frac{p}{2}-1)f^{p-1}\frac{|\nabla
\eta|^2}{\eta^2}. \label{eq-help2}
\end{eqnarray}
Now combining the previous estimates, we have that
\begin{eqnarray*}
0&\ge& f^{p/2-1}\eta^{p-1}\mathcal{L}(Q) \\
&\ge& \frac{2}{n-1} Q^p
+Q^{p-2}\left(\frac{2(p/2-1)^2}{n-1}+\frac{(p-2)p}{2}\right)\langle \nabla u, \nabla \eta\rangle^2 \\
&\quad& +\left(\frac{2(p-1)}{n-1}-2\right)Q^{p-1}\langle \nabla u,
\nabla \eta \rangle -\left((\frac{p}{2}+1)\max
\{p-1, 1\}\right) Q^{p-1}\frac{|\nabla \eta|^2}{\eta}\\
&\quad& -\left(2(n-1)K^2+20(n+p-3)\frac{1+KR}{R^2} +\frac{10\max
\{p-1, 1\}}{R^2}  \right)Q^{p-1}.
\end{eqnarray*}
Since
$$
Q^{p-2}\left(\frac{2(p/2-1)^2}{n-1}+\frac{(p-2)p}{2}\right)\langle
\nabla u, \nabla \eta\rangle^2 \ge
-\left(\frac{2(p/2-1)^2}{n-1}+\frac{(p-2)p}{2}\right)_{-}Q^{p-1}\frac{|\nabla
\eta|^2}{\eta}
$$
and
$$\left(\frac{2(p-1)}{n-1}-2\right)Q^{p-1}\langle \nabla u,
\nabla \eta \rangle\ge -\frac{2\epsilon}{n-1}Q^p -\frac{b_{p,
n}^2(n-1)}{8\epsilon}Q^{p-1}\frac{|\nabla \eta|^2}{\eta}
$$
with $b_{p, n}=\left(\frac{2(p-1)}{n-1}-2\right)$, we have that
\begin{eqnarray*}
0&\ge& f^{p/2-1}\eta^{p-1}\mathcal{L}(Q) \\
&\ge& \frac{2(1-\epsilon)}{n-1} Q^{p}
-\frac{10}{R^2}Q^{p-1}\left(c_{p, n}+\frac{b_{p,
n}^2(n-1)}{8\epsilon}\right)\\
&\quad& -\left(2(n-1)K^2+20(n+p-3)\frac{1+KR}{R^2} +\frac{10\max
\{p-1, 1\}}{R^2} \right)Q^{p-1}
\end{eqnarray*}
where
$$
c_{p,
n}=\left(\frac{2(p/2-1)^2}{n-1}+\frac{(p-2)p}{2}\right)_{-}+\left((\frac{p}{2}+1)\max
\{p-1, 1\}\right).
$$
Here we have used $\frac{|\nabla \eta|^2}{\eta} \le \frac{10}{R^2}$.
Theorem \ref{main1} then follows easily from the above inequality.

\section {Boundary estimate}

Let  $\Omega\subset M$ be an open subset such that $\Omega^{c}$ is
compact and $\partial \Omega$ is $C^{\infty}$. Again, for $p > 1$,
we consider the $p$-harmonic equation:
\begin{equation}
\label{eq:pharmonic}
\widehat{\Lambda}_p(v):=\operatorname*{div}\left(\left|\nabla
v\right|^{p-2}\nabla v\right) = 0
\end{equation}
Any positive solution $v$ of \eqref{eq:pharmonic} gives rise to a
solution $u$ of
\begin{equation}
\label{eq:uequation} \Lambda_p(u) :=
\operatorname*{div}\left(\left|\nabla u\right|^{p-2} \nabla u\right)
- \left|\nabla u\right|^p = 0
\end{equation}
via the relationship $u = (1-p)\log v$.

Our primary objective in this section is to prove the following
boundary estimate, which corresponds to a similar result in
\cite{IH}.

\begin{proposition}
\label{prop:boundarymc} For every $\epsilon > 0$, there exists $p_0
= p(\epsilon) > 1$ such that if $u$ satisfies equation
\eqref{eq:uequation} on $\Omega$ for some $1 < p \leq p_0$, then we
have the estimate
\begin{equation}
\label{eq:boundarymc}
    \left|\nabla u\right| \leq H_{+} + \epsilon
\end{equation}
where $H$ denotes the mean curvature of $\partial\Omega$ and
$H_{+}(x) = \max\left\{ H(x), 0\right\}$.
\end{proposition}

We begin by recording some simple equations.
\begin{lemma}
\label{lem:radialeqs}

Suppose $\phi: (0,\infty) \to \mathbb{R}$ is a smooth function and
$r(x) = d(x, x_0)$  for some fixed $x_0\in M$.  Fix $p \geq 1$. If
$\Sigma = \operatorname*{cut}(x_0)$, then for $x\in M
\setminus\Sigma$, we have the following formula for $w(x) =
\phi(r(x))$
\begin{align}
\nabla_i\nabla_j w &= \phi^{\prime}\nabla_i\nabla_j r +
\phi^{\prime\prime}\nabla_ir \nabla_jr \notag \\
    \widehat{\Lambda}_p(w) &= \left|\nabla w\right|^{p-2}\left(\Delta w
    + (p-2)\frac{\left(\nabla\nabla w \right)\left(\nabla w, \nabla w\right)}
    {\left|\nabla w\right|^{2}}\right)\notag\\
    \label{eq:radialharmoniceq}
                    &=      \left|\phi^{\prime}\right|^{p-2}
                    \left(\phi^{\prime}\Delta r + (p-1)\phi^{\prime\prime}\right)
\end{align}
\end{lemma}

Now we prove a preliminary estimate that will assist us in the proof
of Proposition \ref{prop:boundarymc}.
\begin{proposition}
\label{prop:boundaryball} Suppose $u$ is a solution to
\eqref{eq:uequation} on $\Omega$ for some $1 < p < n$ with $u = 0$
on $\partial\Omega$.  Define
\[
\beta(x) := \sup\left\{\;r > 0\; \left| \; \exists y\in \Omega^c \;
\mbox{with}\; B(y, r) \subset \Omega^c\; \mbox{ and }\; x\in
\partial B(y,r)\; \right. \right\},
\]
\[
\iota = \min_{x\in\Omega^c}\operatorname*{inj}(x)
\]
and put
\[
R = \min\left\{\inf_{x\in\partial\Omega}\beta(x),
\frac{\iota}{3}\right\}.
\]
Choose $K\geq 0$ so that $\operatorname*{Rc}\geq -(n-1)Kg$ on
\[
        S_{2R} := \left\{\;x\in M\;\left|\; d\left(x, \Omega^c\right)\leq 2R\;\right. \right\}
\]
Then
\begin{equation}
\label{eq:boundaryball}
    \left|\nabla u\right|\leq \frac{n-1}{R}\left(1+2\sqrt{K}R\right)
    \left(1-2^{\frac{n-p}{1-p}}\right)^{-1}
\end{equation}
on $\partial\Omega$.
\end{proposition}

We should remark that in the case that $M$ has nonnegative Ricci
curvature, the above estimate can be sharpened to
$$
 \left|\nabla u\right|\leq \frac{n-1}{R},
$$
which generalizes the one proved in \cite{Moser-j}.
\begin{proof}
    Consider the function $\sigma:(0,\infty)\to\mathbb{R}$ defined by
\[
        \sigma(r) := \left(\frac{Re^{-\sqrt{K}(r-R)}}{r}\right)^{\frac{n-1}{p-1}}.
\]
If we put
\[
    \phi(r) := 1 - \alpha \int_{R}^r\,\sigma(\rho)\,d\rho
\]
for $\alpha > 0$ to be determined later, we have $\phi(R) = 1$ and
\begin{equation}
\label{eq:phiprimequotient}
\frac{\phi^{\prime\prime}}{\phi^{\prime}} =
-\left(\frac{n-1}{p-1}\right)\left(\sqrt{K} + \frac{1}{r}\right).
\end{equation}
Now let $u$ be a solution of \eqref{eq:uequation} for some $1 < p <
n$. Fix $x_0\in \partial \Omega$, and choose $y_0\in \Omega^c$ such
that $B(y_0, R) \subset \Omega^c$ and $x_0\in \partial B(y_0, R)$.
Setting $r(x) = d(y_0, x)$ and $\widetilde w(x) = \phi(r(x))$, we
have, from \eqref{eq:radialharmoniceq} and
\eqref{eq:phiprimequotient}
\begin{align}
\widehat{\Lambda}_p(\widetilde{w}) &=
\left|\phi^{\prime}\right|^{p-2}\phi^{\prime}
\left(\Delta r +(p-1)\frac{\phi^{\prime\prime}}{\phi^{\prime}}\right)\notag\\
    \label{eq:showssub}       &= \left|\phi^{\prime}\right|^{p-2}\phi^{\prime}
    \left(\Delta r - (n-1)\left(\sqrt{K} + \frac{1}{r}\right)\right)
\end{align}
away from the cut locus of $y_0$.  Since $\phi^{\prime} < 0$,
applying the Laplacian comparison theorem to \eqref{eq:showssub}
implies that $\widehat{\Lambda}_p(\widetilde{w}) \geq 0$, i.e., that
$\widetilde{w}$ is a subsolution to \eqref{eq:pharmonic} on $B(y_0,
2R)\setminus\{y_0\}$ (recall that $3R \leq
\operatorname*{inj}(y_0)$). Now choose $\alpha$ to be
\begin{equation}
\label{eq:alphaeq}
    \alpha := \left(\int_{R}^{2R}\,\sigma(\rho)\,d\rho\right)^{-1}
\end{equation}
Then we have $\widetilde{w}(x) = \phi\left(2R\right) \equiv 0 $ on
$\partial B(y_0, 2R)$ and $\widetilde{w}(x) \leq1$ on
$\partial\Omega$. Extend $\widetilde{w}$ to a subsolution on all of
$M\setminus\{y_0\}$ by
\[
    w(x) := \begin{cases}
            \widetilde{w}(x) &\mbox{if}\;\; x\in B(y_0, 2R)\setminus\{y_0\}\\
            0 &\mbox{if}\;\;x\in B(y_0, 2R)^c.
        \end{cases}
\]
Since the function $v = \exp(u/(1-p))$ satisfies $v \equiv 1$ on
$\partial \Omega$ and $v > 0$ in $\Omega$, it follows from the
comparison principle
 that $w \leq v$ on $\Omega$.  In particular,
we have $0 < w \leq v$ on $\overline{B(y_0, (3/2)R)}\cap\Omega$,
Thus if we define $z$ by $z(x) :=(1-p)\log w(x)$, $z$ is a
non-negative supersolution to \eqref{eq:uequation} on the annular
region
\[
    A := A_{R, \frac{3R}{2}} := \overline{B\left(y_0, 3R/2\right)}\setminus B\left(y_0, R\right) ,
\]
satisfying $z(x) > 0$. If we extend $u$ to be identically $0$ in
$\Omega^c$, we have $z\geq u$ on $\partial A$, and hence in all of
$A$ by the comparison principle.  In particular, at $x_0$, we have,
for all smooth paths $\gamma:(-\epsilon, \epsilon)\to M$ with
$\gamma(0) = x_0$ and $\gamma^{\prime}(0) = -\nu(x_0)$ (where $\nu$
is the unit normal to $\partial\Omega$ at $x_0$ pointing into
$\Omega^c$),
\begin{equation}
\label{eq:boundderivcomp}
    \frac{u(\gamma(h)) - u(x_0)}{h}  = \frac{u(\gamma(h))}{h} \leq \frac{z(\gamma(h))}{h}
        = \frac{z(\gamma(h)) - z(\gamma(x_0))}{h}.
\end{equation}
Hence
\[
    |\nabla u|(x_0) \leq |\nabla z|(x_0).
\]
It remains, then, just to estimate the gradient of $z$ at $x_0$.

Recall from the definition of $z$ that
\begin{align}
    |\nabla z|(x_0) &= (p-1) \frac{|\phi^{\prime}(R)|}{\phi(R)}\notag\\
\label{eq:alpha}                                 &= (p-1)\alpha.
\end{align}
Integrating by parts, we see
\begin{equation*}
    \int_{R}^{2R}\,\sigma(\rho)\,d\rho =\left(\frac{p-1}{n-1}\right)\left(R-2R\sigma(2R)\right) -
    \sqrt{K}\frac{n-1}{n-p}\int_R^{2R}\,\rho\sigma(\rho)\,d\rho.
\end{equation*}
So
\begin{align*}
    \left(1 + 2\left(\frac{n-1}{n-p}\right)\sqrt{K}R\right)\int_R^{2R}\,\sigma(\rho)d\rho
     &\geq \int_{R}^{2R}\,\left(1+ \left(\frac{n-1}{n-p}\right)\sqrt{K}\rho\right)\sigma(\rho)d\rho\\
     &= R\left(\frac{p-1}{n-p}\right)\left(1-2\sigma(2R)\right).
\end{align*}
Since
\[
    2\sigma(2R) \leq 2^{\frac{n-p}{1-p}},
\]
we have
\begin{align*}
    \alpha & = \left(\int_R^{2R}\,\sigma(\rho)\,d\rho\right)^{-1}\\
                                &\leq \frac{n-p}{(p-1)R}\left(1  +
                                 2\frac{n-1}{n-p}\sqrt{K}R\right)
                                 \left(1 - 2^{\frac{n-p}{1-p}}\right)^{-1},
\end{align*}
and the desired inequality \eqref{eq:boundaryball} follows from
\eqref{eq:alpha}.
\end{proof}

Now we turn to the proof of Proposition \ref{prop:boundarymc}.  We
use the device in Lemma 3.4 of \cite{IH} to obtain a supersolution
for the solution $u$, using the rudimentary control near the
boundary we have already obtained in the previous Proposition in
concert with the interior gradient estimates of Theorem \ref{main1}
to ensure that the growth of our supersolution is sufficiently
rapid.
\begin{proof}[Proof of Proposition \ref{prop:boundarymc}]
Let $H$ denote the mean curvature of $\partial \Omega$ and set
$H_{+} = \max\{H, 0\}$. Choose a smooth function $\widetilde{w}$
with $\widetilde{w}\equiv 0$ on $\partial\Omega$ such that
\begin{equation}
\label{eq:estonbdry}
    H_{+} < \frac{\partial \widetilde{w} }{\partial \nu} \leq H_{+} + \epsilon
\end{equation}
on $\partial\Omega$.  Then $\left|\nabla \widetilde{w}\right| > 0$
and
\begin{align*}\
    \Lambda_1(\widetilde{w}) &= \frac{\Delta \widetilde{w}}{\left|\nabla\widetilde{w}\right|}
        - \frac{\left(\nabla\nabla \widetilde{w}\right)\left(\nabla\widetilde{w},
         \nabla\widetilde{w}\right)}
                    {\left|\nabla \widetilde{w}\right|^{3}} - \left|\nabla\widetilde{w}\right|\\
                        &=H - \frac{\partial \widetilde{w}}{\partial \nu}\\
                        &< 0
\end{align*}
on $\partial \Omega$; indeed, for suitably small $\delta > 0$, both
inequalities hold strictly on $U_{\delta}$, which we define to be
the components of $\left\{0 \leq \widetilde{w} < \delta\right\}$
containing $\partial\Omega$. Now introduce
\[
    w: = \frac{\widetilde{w}}{1 - \widetilde{w}/\delta}.
\]
Then also $w\equiv 0$ on $\partial\Omega$ and $w(x)\to \infty$ as
$x$ approaches $\partial U_{\delta}\setminus \partial{\Omega}$.
Moreover, a computation shows
\[
    \Lambda_1(w) = \Lambda_{1}(\widetilde{w}) +
        \left(1 - \frac{1}{\left(1- \widetilde{w}/\delta\right)^{2}}\right)
        \left|\nabla \widetilde{w}\right|,
\]
so $\Lambda_1(w) < 0$ on $U_{\delta}$ as well.  By
Proposition \ref{prop:boundaryball}, and the interior estimates in
Theorem \ref{main1}, for any precompact open set $V$ containing
$\Omega^c$, there exists a constant $C= C(K_M, V)$ depending on the
geometry of $\partial\Omega$ and the lower bound of the sectional
curvature in $V$
 (but independent of $p$!) such that any $u$ satisfying \eqref{eq:uequation}
with $u\equiv 1$ on $\partial\Omega$ satisfies $u(x) \leq C$ on $V$.
In particular, if we choose such a $C = C(\delta)$ for $U_{\delta}$,
we see that $w \geq u$ for any such $u$ on $\partial U_{\delta}$.
More precisely we can define $ \widetilde U_{C}$ to be the
components of $\{0\le w\le C\}$ in $U_{\delta}$. Clearly $w\ge u$ on
$\partial \widetilde U_{C+1}$.

But we may also compute
\[
        \Lambda_p(w) = \left|\nabla w\right|^{p-1}\left( \Lambda_1(w)
            + (p-1)\frac{\left(\nabla\nabla w\right)\left(\nabla w, \nabla w\right)}
            {\left|\nabla w\right|^{3}}\right),
\]
which shows that choosing $p_0$ sufficiently close to $1$, depending
on $C(\delta)$ (and implicitly on our original $\epsilon$), as well
as bounds on our (fixed) barrier $w$, we have that
$\Lambda_p(w) < 0$ for all $1 < p \leq p_0$ on $\widetilde
U_{C+1}$. Therefore, it follows from the comparison principle for
equation \eqref{eq:uequation} that $u\leq w$ on
$\overline{\widetilde U}_{C+1}$. Hence, recalling
\eqref{eq:estonbdry}, we have
\[
    \frac{\partial u}{\partial \nu}\leq \frac{\partial w}{\partial \nu}
    = \frac{\partial\widetilde{w}}{\partial \nu} \leq H_{+} + \epsilon
\]
on $\partial\Omega$ as claimed.
\end{proof}

\section{$1/H$-flow}

The $1/H$ (or {\it inverse mean curvature}) flow is a parabolic evolution equation
for hypersurfaces.  Given an initial embedding $X_0:N^{n-1}\to M^n$, the flow may be defined
parametrically by $N_t= X_t(N)$ where
\[
   \pd{X_t}{t} = \frac{1}{H}\nu.
\]
Here $\nu$ denotes the outward normal.  Alternatively, the flow may be defined in terms of
a level set formulation, in which case $N_t$ is given by $\partial\{u < t\}$ for a $u$ satisfying
(\ref{1/h}). A theory of weak solutions to \eqref{1/h} was established in
\cite{IH}, based on a variational principle involving the functional
$$
J_u(w; K)\doteqdot \int_K \left(|\nabla w|+w|\nabla u|\right)\, d\mu
$$
for any precompact subset $K\subset \Omega$; the reader is encouraged to consult
this paper for further details and the motivation behind this theory.
In this section, our interest is in the general problem of the existence of weak solutions
to \eqref{1/h}, particularly in the case in which $N_0$ is the compact
boundary of an end $\Omega$ of $M$.

\begin{definition} A function $u\in
C_{\operatorname{loc}}^{0,1}(\Omega)$ is called a weak solution of
(\ref{1/h}) if for every precompact set $K\subset \Omega$ and every
$w\in C_{\operatorname{loc}}^{0,1}(\Omega)$ with $w=u$ in $\Omega
\setminus K$, the inequality
\begin{equation}\label{1/h-weak}
J_u(u; K) \le J_u(w; K)
\end{equation}
holds. A weak solution is called proper if $\lim_{x\to \infty} u(x)
=+\infty$.
\end{definition}

In \cite{Moser-j}, the functional
$$
J_{u}^p(w; K)\doteqdot \int_K \left(\frac{1}{p}|\nabla w|^p
+w|\nabla u|^p\right)\, d\mu
$$
was introduced for every precompact set $K\subset \Omega$. It was
shown that if $v$ is $p$-harmonic and $u=-(p-1)\log v$, then
\begin{equation}\label{p-energy}
J^p_u(u; K)\le J_u^p(w; K)
\end{equation}
for every $w\in W^{1, p}_{\operatorname{loc}}(\Omega)$ satisfying
$w=u$ in $\Omega\setminus K$.
By a compactness-type argument
which originated in \cite{IH}, it was also shown (p. 82 of \cite{Moser-j})  that
once one has a uniform estimate on $|\nabla u|$ that is independent of $p$
as $p\to 1$, one can obtain a weak solution to (\ref{1/h}) from the limit of $u$ as $p\to 1$.
To ensure that the solution one obtains is proper, one needs further estimates.  In the case
of \cite{Moser-j}, these are furnished by comparison with explicit solutions to
\eqref{eq-u} on $\mathbb{R}^n$.  In our case, we derive suitable estimates from
results of \cite{Ho} (see also
\cite{LT}).

First we need to recall the notions of $p$-Green's functions and
$p$-nonparabolicity. Let $G(x, z)$ be the Green's function centered
at $z\in M$.  One can construct $G$ via a compact
exhaustion as in \cite{LT} for the case $p=2$; this
was done in \cite{Ho}. To ensure that the Green's
function so obtained is positive, one must impose conditions on the end $\Omega$.
For the statement of these conditions, we will need to recall
 the notion of {\it homogeneous} ends from \cite{Ho}, and
the associated {\it volume doubling} (\textbf{VD}) condition and {\it
weak Neumann Poincar\'e} (\textbf{WNP}) inequality.

Let $o\in M$ be a fixed point.
A manifold $M$
(or an end $\Omega$ of $M$)  is said to have  property (\textbf{VD})
if there exists $C_1>0$ such that for any $x\in M $ (or $\in
\Omega$)
$$
V(x, 2r)\le C_1 V(x, r).
$$
Here $V(x, r)$ denotes the volume of the ball $B(x, r)$.  It is said to have
property (\textbf{WNP}) if there exists $C_2(p)>0$ such that
$$
\frac{1}{V(x, r)}\int_{B(x, r)} |u-\bar{u}|\le C_2r
\left(\frac{1}{V(x, r)}\int_{B(x, 2r)}|\nabla u|^p\right)^{1/p}
$$
where $\bar{u}=\frac{1}{V(x, r)}\int_{B(x, r)}u$. For any subset
$\Omega$ we define $\Omega(t) \doteqdot \Omega \setminus B(o, t)$.

\begin{definition} Assume that  (\textbf{VD}) and
(\textbf{WNP}) hold on $\Omega$. The end $\Omega$ is called
$p$-homogeneous  if there exists $C_3$ such that
\begin{equation}\label{vc}
V(\Omega\cap B(o, t))\le C_3 V(x, \frac{t}{8})
\end{equation}
for every $x\in \partial \Omega(t)$. $\Omega$ is called uniformly
homogeneous if $C_2(p)\le C_2'$ for some $C_2'$ independent of $p$ as
$p\to 1$.
\end{definition}
As in \cite{LT}, we denote the above volume comparison condition
(\ref{vc}) by (\textbf{VC}).

\begin{definition} A end $\Omega$ is called $p$-nonparabolic, for some $p>1$, if
\begin{equation}\label{pnp}
\int_1^\infty \left(\frac{t}{V(\Omega\cap B(o,
t))}\right)^{1/(p-1)}\, dt < \infty.
\end{equation}
\end{definition}

Tracing the proof of Proposition 5.7 in \cite{Ho} we have the
following estimate on $p$-Green's function. Note that the uniformly
homogeneous condition implies a uniform Harnack constant for positive
$p$-harmonic functions in the Harnack inequality (2.23) of \cite{Ho}.

\begin{proposition} Assume that $\Omega$ is uniformly homogeneous and $p$-nonparabolic. Then
there exists a positive Green's function $G(x, z)$ on $M$.
Furthermore, there exists a positive constant $C_4$ independent of $p$
(as $p\to 1$) such that
\begin{equation}\label{g-est}
G(x, z)\le \frac{C_4}{(p-1)^2} \int_{2r}^\infty
\left(\frac{t}{V(\Omega\cap B(o,t))}\right)^{1/(p-1)}\, dt
\end{equation}
for every $x\in \Omega(r)$.
\end{proposition}

The factor $\frac{1}{(p-1)^2}$ comes from applying Young's
inequality in Lemma 5.6 of \cite{Ho}. In \cite{LT} and \cite{Ho}, it can be seen
that, by solving the Dirichlet problem on a compact exhaustion, one can obtain a positive
$p$-harmonic  function $v$ with $v=1$ on $\partial \Omega$ and
$\lim_{x\to \infty} v(x)=0$, provided the end is
$p$-nonparabolic. It is easy to see from the proof of \cite{Ho} that
the above estimate (\ref{g-est}) also holds for such $v$. This is
what is needed for our purposes. Alternatively, one can obtain such an estimate from
(\ref{g-est}), using the Green's function as an upper barrier and applying the comparison principle.

Combining the gradient estimates of the previous sections and the
above result we have the following existence theorem.

\begin{theorem}\label{exist} Let $M$ be a complete Riemannian manifold. Assume
that $\Omega$ is a uniformly homogeneous end and is
$p_0$-nonparabolic for some $p_0>1$. Additionally, assume that
$\lim_{r\to \infty}\mathcal{V}(r)=0$ where
$$\mathcal{V}(r)=\sup_{2r\le t< \infty}\frac{t}{V(\Omega \cap B(o,
t ))}.
$$
Then there exists a proper solution $u$  to (\ref{eq-u}) on $\Omega$
with $u=0$ on $\partial \Omega$.
\end{theorem}

Before turning to the proof of the theorem, let us first comment on its assumptions.
The uniformly homogeneous assumption holds in particular for manifolds of
so-called asymptotically nonnegative sectional curvature.  These are manifolds
$M$ for which there exists a continuous non-increasing function $k(t):
[0, \infty)\to [0, \infty)$ such that $K_M \ge -k(t)$ and
$\int_0^\infty tk(t)\, dt <\infty$.  The reader is referred to \cite{LT}
and \cite{H} for the verification of the fact that on such manifolds
the first two assumptions of the theorem are satisfied. Notice that
in this case Theorem \ref{main1} implies that $|\nabla u|(x)\to 0$
as $x\to \infty$. Also notice that if
$$
V(\Omega \cap B(o, t ))\ge \delta t^{1+\epsilon}
$$
for some $\delta>0$ and $\epsilon>0$, then $\Omega$ is
$p$-nonparabolic for sufficient small $p$, and satisfies $\lim_{r\to
\infty}\mathcal{V}(r)=0$. In particular the above theorem can be
applied to the  asymptotically {\it locally} Euclidean (ALE)
manifolds, which include the cases considered in \cite{IH}. We refer
the reader to \cite{LT} and \cite{Ho} for other  {\it partially
homogenous} examples, including the manifolds with finite first Betti number
and nonnegative
Ricci curvature outside a compact set.

In a sense, the extra volume growth condition ($\lim_{r\to \infty}\mathcal{V}(r)=0$)  imposed in the theorem is
optimal. Let $(M, g)$ be Hamilton's
2-dimensional  `cigar' \cite{H}, i.e, $M=\R^2$ and $g=ds^2+\tanh^2
s d\theta^2$. Then $M$ has linear volume growth and nonnegative Ricci curvature, and is homogeneous,
however
$$
u\doteqdot\ln \left(\frac{\tanh s}{\tanh 1}\right).
$$
is a solution to \eqref{eq-u} on $\{u\ge 0\}\subset M$ but is not proper.
Moreover, we can show that in fact $M$ admits no proper solution.  For suppose that $u$ is such a solution.
Let $\varphi$ be a cut-off function. Integration by parts yields that
$$
\int_M \varphi |\nabla u|\, d\mu \le \int_M |\nabla \varphi|\, d\mu.
$$
Using the fact that $(M, g)$ is of linear volume growth, by choosing suitable $\varphi$ we can conclude that
$$
\int_M |\nabla u|\, d\mu \le C
$$
for some $C$ depending on the geometry of $M$.
On the other hand,
$$
\int_0^{2\pi} u(R, \theta) d\theta \to \infty$$
as $R\to \infty$, since $u$ is proper. In particular, we have that
$$
\int_1^{R}\int_0^{2\pi} u_s(s, \theta)\, d\theta, ds\to \infty.
$$
But, if we denote the point where $s=0$ by $o$,
it is easy to see that there exists $C_1$ such that
$$
\int_1^{R}\int_0^{2\pi} u_s(s, \theta)\, d\theta \le C_1\int_{B(o, R)\setminus B(o,1)} |\nabla u|\, d\mu \le C_1C,
$$
 which is a contradiction.
This shows that the volume condition $\lim_{r\to \infty}\mathcal{V}(r)\to 0$ is necessary for
the existence of a proper solution. Note that the example can be adapted to any dimension.

\begin{proof} (of Theorem \ref{exist}.)
We first construct $p$-harmonic functions $v^{(p)}$ for $p\le p_0$
with $v^{(p)}=1$ on $\partial \Omega$ and $v^{(p)}(x)\to 0$ as $x\to
\infty$. This can be done by solving the Dirichlet problem on a
compact exhaustion and then taking a limit. In view of the
regularity result of \cite{Lewis}, the $p$-harmonic function
$v^{(p)}$ will be $C^{1, \alpha}$. On the set where $|\nabla u|\ne
0$, elliptic regularity theory implies that both Theorem \ref{main1}
and the boundary estimate of Section 3 are valid and can be applied.
The existence of a non-trivial limit is then ensured by the estimate below. For
$p\le p_0$ it is easy to see that
$$
\left(\int_{2r}^\infty \left(\frac{t}{V(\Omega\cap B(o,
t))}\right)^{1/(p-1)}\, dt\right)^{p-1} \le
\mathcal{V}(r)^{\frac{p_0-p}{p_0-1}}\left(\mathcal{A}_{p_0}(r) \right)^{p-1}
$$
where
$$
\mathcal{A}_{p_0}(r)=\int_{2r}^\infty \left(\frac{t}{V(\Omega\cap B(o,
t))}\right)^{1/(p_0-1)}\, dt.
$$
 Hence $\Omega$ is also $p$-nonparabolic for $p\le p_0$. This
ensures that there exist $v^{(p)}$ satisfying
$$
v^{(p)}(x)\le
\frac{C_4}{(p-1)^2}\int_{2r(x)}^\infty\left(\frac{t}{V(\Omega\cap
B(o, t))}\right)^{1/(p-1)}\, dt.
$$
Combining the previous inequalities, we have that for $r(x)>>1$
\begin{eqnarray*}
u^{(p)}(x)&\ge& 2(p-1)\log (p-1) -(p-1)\log C_4-(p-1)\log
(\mathcal{A}_{p_0}(r(x))) \\
&\quad&-\frac{p_0-p}{p_0-1}\log (\mathcal{V}(r(x))).
\end{eqnarray*}
By the gradient estimates in the previous sections we can conclude
that $u^{(p)}$ converges locally uniformly to a limit function
$u(x)\in C_{\operatorname{loc}}^{0,1}(\Omega)$. Moreover, $u(x)$
satisfies that
$$
u(x)\ge-\log (\mathcal{V}(r(x)))
$$
for $r(x)>>1$.
 By the compactness argument from \cite{Moser-j}, Theorem \ref{main1}, and
Proposition \ref{prop:boundarymc}, we can conclude that $u$ is a
weak proper solution to (\ref{eq-u}).
\end{proof}

\section{The $p$-Laplacian heat equation}

Motivated by Theorem  \ref{main1}, we consider smooth solutions to
the parabolic analog of \eqref{eq-v}, namely
\begin{equation}\label{eq:p-par}
    \pd{v}{t} =  \operatorname{div}(|\nabla v|^{p-2}\nabla v)
\end{equation}
for $p > 1$.  This nonlinear evolution equation is the gradient flow for the $p$-energy functional
\[
    E_p(v) = \int_{M}|\nabla v|^p d\mu
\]
and has been studied rather extensively -- see, for example,
\cite{Ba},  \cite{V1}, and the references therein. For a given smooth solution $v$ of \eqref{eq:p-par}, it
will be useful for us to consider the linearization of the operator $\widehat{\Lambda}_p$ at $v$, given by
\begin{equation*}
    \widehat{\mathcal{L}}(\psi) = \operatorname{div}(h^{p/2-1}A(\nabla \psi))
%
\end{equation*}
Here $h = |\nabla v|^2$ and $A$ is the tensor introduced in Section 2, namely
$$
A_{ij}=\delta_{ij} +(p-2) \frac{v_i v_j}{h}.
$$

The main result in this section is essentially a consequence of the following
calculation.
\begin{lemma}\label{lem:heatopcomp}
Suppose $ v:M\times [0, T)\to \R$ is a smooth, positive solution to
\eqref{eq:p-par} with $p >1$. For any $\alpha
> 0$, define
\[
    F_{\alpha} := \frac{|\nabla  v|^p}{ v^2} - \alpha\frac{v_t}{ v}.
\]
Then, on the region $|\nabla v|>0$, the following estimate holds
\begin{align}
\label{eq:mainevol}
\begin{split}
    \left(\widehat{\mathcal{L}} - \pd{}{t}\right) F_{\alpha}
    &= ph^{p-2}\left(\left|\frac{v_{ij}}{v} - \frac{1}{p-1}\frac{v_i v_j}{ v^2}
                 \right|^2_A+\frac{R_{ij}v_iv_j}{v^2}\right)
   +(\alpha -1)(p-2)\frac{v_t^2}{v^2} \\
    &\phantom{=}\quad\quad+ (p-2)F_{1}^2
     -2(p-1)\frac{h^{p/2-1}}{v}\langle \nabla F_{\alpha}, \nabla  v\rangle.
\end{split}
\end{align}
In particular, if $\operatorname{Rc} \geq - K g$ for some $K > 0$, we have
\begin{align}
\begin{split}
\label{eq:mainevol2}
        \left(\widehat{\mathcal{L}} - \pd{}{t}\right) F_{\alpha} &\geq
        \left(\frac{p + n(p-2)}{n}\right)F_1^2
            +(\alpha -1)(p-2)\frac{v_t^2}{ v^2}\\
            &\qquad -pK\frac{h^{p-1}}{v^2}-2(p-1)\frac{h^{p/2-1}}{v}
            \langle \nabla F_{\alpha}, \nabla  v\rangle.
\end{split}
\end{align}  Here, for a two-tensor $T$, we write $|T|_A^2 := A^{ik}A^{jl}T_{ij}T_{kl}$.

\end{lemma}
\begin{proof} First we have the following formula resembling Lemma
\ref{bochner1}.
\begin{lemma}\label{bochner2} Let $h=|\nabla v|^2$. Then
\begin{equation}
\left(\frac{\partial}{\partial t}-\widehat{\mathcal{L}}\right)h
=-2h^{p/2-1}(|\nabla\nabla v|^2+R_{ij}v_iv_j)-(\frac{p}{2}-1)h^{p/2-2}|\nabla
h|^2.
\end{equation}
\end{lemma}
The proof of the above identity is a straightforward calculation and very similar to
that of Lemma \ref{bochner1}.  Hence we leave the details to
interested readers.

By the definition of $\widehat{\mathcal{L}}$, we have that
    \begin{equation}\label{eq:pdtevol}
        \left(\widehat{\mathcal{L}} - \pd{}{t}\right)v_t = 0,
\end{equation}
and with the help of Lemma \ref{bochner2}, we find that
    \begin{equation}\label{eq:normpevol}
        \left(\widehat{\mathcal{L}} -\pd{}{t}\right)h^{p/2} =
        ph^{p-2}\left[|\nabla\nabla  v|^2_A +
                            R_{ij}v_i v_j\right].
    \end{equation}
Using the general formula
    \begin{align*}
    \begin{split}
            \widehat{\mathcal{L}}\left(\frac{Q}{ v^{\gamma}}\right) &=
                \frac{1}{ v^{\gamma}}\widehat{\mathcal{L}}(Q) -
                \gamma \frac{Q}{ v^{\gamma+1}}\widehat{\mathcal{L}}( v)\\
                &\qquad- \frac{2\gamma(p-1)}{ v^{\gamma +1}}h^{p/2-1}
                \langle\nabla Q, \nabla  v\rangle
                    +\gamma(\gamma + 1)(p-1)h^{p/2}\frac{Q}{ v^{\gamma + 2}},
    \end{split}
    \end{align*}
    and the identity
    \[
            \left(\widehat{\mathcal{L}} -\pd{}{t}\right) v = (p-2)v_t
    \]
    together with \eqref{eq:normpevol} and \eqref{eq:pdtevol}, we then obtain
    \begin{align}
        \begin{split}\label{eq:normpphievol}
            \left(\widehat{\mathcal{L}} -\pd{}{t}\right)\frac{h^{p/2}}{ v^2}&=
                \frac{p}{ v^2}h^{p-2}\left[|\nabla\nabla v|^2_A +
                R_{ij}v_iv_j)\right]
                -2(p-2)h^{p/2}\frac{v_t}{v^3}\\
                &\qquad -4p(p-1)\frac{h^{p-2}}{v^3}v_{ij}v_iv_j
                + 6(p-1)\frac{h^{p}}{v^4}
        \end{split}
    \end{align}
    and
    \begin{align}
    \begin{split}\label{eq:pdtphievol}
        \left(\widehat{\mathcal{L}} -\pd{}{t}\right)\frac{v_t}{v} &=
        -(p-2)\frac{v_t^2}{v^2}
            -2(p-1)\frac{h^{p/2-1}}{v^2}\left\langle\nabla v_t, \nabla  v\right\rangle \\
            &\quad\quad+ 2(p-1)\frac{h^{p/2}}{v^3}v_t.
    \end{split}
    \end{align}
Now,
    \begin{align*}
    \begin{split}
        \nabla_k F_{\alpha} = ph^{p/2-1}\frac{v_{ik}v_i}{v^2}
        - 2\frac{v_k}{ v^3}
        -\alpha\frac{\nabla_kv_t}{v} + \alpha\frac{v_t}{v^2} v_k
    \end{split}
    \end{align*}
    so, after multiplying both sides by $\alpha$, \eqref{eq:pdtphievol} becomes
    \begin{align*}
    \begin{split}
            \left(\widehat{\mathcal{L}} -\pd{}{t}\right)\left(\alpha \frac{v_t}{v}\right)
            &=-\alpha(p-2)\left(\frac{v_t}{v}\right)^2
            - 2p(p-1)\frac{v_{ij}v_iv_j}{v^3}\\
            &\quad\quad +4(p-1)\frac{h^p}{ v^4}
            +2(p-1)\frac{h^{p/2-1}}{v} \langle \nabla F_{\alpha}, \nabla  v\rangle.
    \end{split}
    \end{align*}
    Combining this with \eqref{eq:normpphievol}, we have
    \begin{align}\label{eq:fevol}
    \begin{split}
        \left(\widehat{\mathcal{L}} -\pd{}{t}\right)F_{\alpha} &=
            p\frac{h^{p-2}}{v^2}\left[|\nabla\nabla v|^2_A + R_{ij}v_iv_j\right]
                -2(p-2)\frac{h^{p/2}}{v^3}v_t\\
                &\quad\quad-2p(p-1)\frac{h^{p-2}}{v^3}v_{ij}v_iv_j
                 + 2(p-1)\frac{h^p}{v^4}\\
                &\quad\quad\quad\quad+\alpha(p-2)\frac{v_t^2}{ v^2}
                - 2(p-1)\frac{h^{p/2-1}}{v}\langle \nabla F_{\alpha}, \nabla v\rangle.
    \end{split}
    \end{align}
    Equation \eqref{eq:mainevol} then follows from \eqref{eq:fevol} and the identity
    \begin{equation*}
       ph^{p-2} \left|\frac{v_{ij}}{v}-\frac{1}{p-1}\frac{v_iv_j}{v^2}\right|^2_A
        = p\frac{h^{p-2}}{v^2} |\nabla\nabla v|_A^2
    -2p(p-1)\frac{h^{p-2}}{v^3}v_{ij}{v_iv_j} + p \frac{h^p}{v^4}.
    \end{equation*}
    For \eqref{eq:mainevol2}, we observe that
   $$
        F_{1} = h^{p/2-1}
        \tr_A\left(\frac{1}{(p-1)}\frac{v_i v_j }{ v^2} - \frac{v_{ij}}{
        v}\right)$$
where $$ \tr_A\left(\frac{1}{(p-1)}\frac{v_i v_j}{ v^2}
- \frac{v_{ij}}{
        v}\right)= A_{ij}\left(\frac{1}{(p-1)}\frac{v_i v_j}{ v^2}
- \frac{v_{ij}}{ v}\right),
$$
    so that
    \[
        ph^{p-2}\left|\frac{1}{(p-1)}\frac{v_iv_j}{ v^2}
        - \frac{v_{ij}}{ v}\right|^2_A\geq
        \frac{p}{n}F_1^2
    \]
    by the standard inequality $n|T|^2_A \geq (\tr_{A}T)^2$ for a two-tensor $T$.
\end{proof}

    In the case $\alpha = 1$, it is convenient to consider the expression of the above equations in terms of the ``pressure'' quantity
\begin{equation}\label{eq:pressure}
    \phi \doteqdot \frac{p-1}{p-2}v^{\frac{p-2}{p-1}}.
\end{equation}
Then
\[
    \nabla\nabla \phi = v^{\frac{p-2}{p-1}}\left(\frac{\nabla\nabla v}{v} - \frac{1}{p-1}\frac{\nabla v\otimes\nabla v}{v^2}\right)
\]
and
\[
    \widehat{\Lambda}_p(\phi) = |\nabla\phi|^{p-2}\tr_A(\nabla\nabla \phi) = -F_1
\]
so that in the case $\alpha =1$, the identity \eqref{eq:mainevol} has the equivalent form
\begin{align}
\label{eq:mainevolp}
\begin{split}
    \left(\pdt - \widehat{\mathcal{L}}\right)\widehat{\Lambda}_p(\phi)
        &\geq p|\nabla \phi|^{2(p-2)}\left(|\nabla\nabla\phi|_{A}^2 +
R_{ij}\phi_i\phi_j\right)
                    + (p-2)\widehat{\Lambda}_p(\phi)^2\\
        &\qquad -2(p - 1) \left\langle \nabla\widehat{\Lambda}_p(\phi), \nabla \phi\right\rangle.
\end{split}
\end{align}

When $\operatorname{Rc}\geq 0$, \eqref{eq:mainevol2} suggests,
by way of the maximum principle, the global estimate
\begin{equation}\label{eq:globalp}
    \widehat{\Lambda}_p(\phi) = \frac{v_t}{v} - \frac{|\nabla v|^p}{v^2} \geq -\frac{n \beta}{t}
\end{equation}
for $\beta = 1 /(p + n(p-2))$ and $p > 2n/(n+1)$.  Such an estimate, analogous to the differential Harnack estimate in \cite{LY} for the heat equation,
 would indeed be sharp in view of the explicit source-type solutions (see, e.g., \cite{Ba})
 \[
        H_p(x,t) = \frac{1}{t^{n\beta}}\left(1 + \frac{\beta^{\frac{1}{p-1}}(2-p)}{p}
            \left(\frac{|x|}{t^{\beta}}\right)^{\frac{p}{p-1}}\right)_{+}^{\frac{p-1}{p-2}}
 \]
to equation \eqref{eq:p-par} on $\mathbb{R}^n$ for which \eqref{eq:globalp} is an equality.  However, \eqref{eq:p-par} is degenerate where $|\nabla v| = 0$, and our calculations, carried out
in the region $|\nabla v|> 0$, are thus as yet insufficient to draw such a conclusion.  Following \cite{EV2}, we therefore consider a family of strictly parabolic equations which approximate
\eqref{eq:p-par} and by proving analogous estimates for the corresponding quantities of the solutions to the approximate equations, we may obtain an estimate of the above form via a limiting
procedure.  The precise statement of the result, which was obtained in \cite{EV2} in the case $M=\mathbb{R}^n$ (albeit for a broader class of solutions), is the following.
We restrict our attention to smooth solutions.
\begin{theorem}
    Suppose $(M^n, g)$ is a complete Riemannian manifold with nonnegative Ricci curvature, and suppose
    $v$ is a smooth, nonnegative solution to \eqref{eq:p-par} with $p > 2n/(n+1)$ for $t\in[0, \Omega)$.
    Then, for all  $t\in (0,\Omega)$, one has
    \begin{equation}
    \label{eq:par-global}
        \frac{|\nabla  v|^p}{ v^2} - \frac{1}{ v}\pd{ v}{t} \leq \frac{n\bar{\beta}}{t}.
    \end{equation}
    where
\[
    \bar{\beta} = \begin{cases} \beta\quad\mbox{if}\quad \frac{2n}{n+1} < p < 2\\
                   (p-1)\beta\quad{if}\quad{p > 2}
    \end{cases}
\]
\end{theorem}
\begin{remark}
    In the case $M = \mathbb{R}$, the above result -- with the sharp constant $\beta = 1/2(p-1)$ -- was obtained in \cite{EV1} for all $p > 1$.
\end{remark}
\begin{proof}

The argument follows almost exactly as in that of Section 2 of \cite{EV2},
where the equation satisfied by the pressure $\phi = (p-1)/(p-2)v^{(p-2)/(p-1)}$, namely,
\[
        \pd{\phi}{t} = \frac{p-2}{p-1}\phi \operatorname{div}(|\nabla \phi|^{p-2}\nabla \phi) + |\nabla\phi|^p,
\]
is approximated by the equation
\begin{equation}\label{eq:p-parapprox}
        \pd{\phi_{\e}}{t} = \frac{p-2}{p-1}\phi_{\e} \operatorname{div}(\varphi_{\e}(|\nabla \phi_{\e}|)\nabla
\phi_{\e}) + \psi_{\e}(|\nabla \phi_{\e}|)
\end{equation}
for judicious choices of $\varphi_{\e}$ and $\psi_{\e}$ depending on $\epsilon > 0$. For completeness, we outline the argument below.

First, take $\zeta_{\e}\in C^{\infty}([0,\infty))$ to satisfy
\[
    \zeta_{\e}(r) = \begin{cases} p-1 & \mbox{if}\quad r \geq \e\\
                               1 &\mbox{if}\quad r \in [0, a_{\e}]
               \end{cases}
\quad\quad
\mbox{where}\quad\quad
      a_{\e} = \begin{cases}
                     \e(p-1)^{1/\e}/2 &\mbox{if}\quad p < 2\\
                     \e(p-1)^{-1/\e}/2 &\mbox{if}\quad p > 2.
               \end{cases}
\]
On the interval $(a_{\e}, r)$, choose $\zeta_{\e}$ to be non-decreasing if $p < 2$,
non-increasing if
$p > 2$, and to satisfy $|r\zeta_{\e}^{\prime}(r)| \leq \e \zeta_{\e}(r)$.  Then let $\varphi_{\e}$
satisfy
\[
     r\frac{\varphi_{\e}^{\prime}(r)}{\varphi_{\e}(r)} = \zeta_{\e}(r) - 1,
\]
so that $\varphi_{\e}(r)$ is constant for small $r$ and equal to $r^{p-2}$ for $r \geq \e$.
Finally, define
\[
     \psi_{\e}(r) = \frac{p}{p-1}\left(r^2\varphi_{\e}(r) - \int_0^r s \varphi_{\e}(s)\,ds\right).
\]
With these choices, \eqref{eq:p-parapprox} is strictly parabolic and we obtain a smooth
solution $\phi_{\e}$ for each $\epsilon > 0$.

Next, we derive an equation analogous to \eqref{eq:mainevolp} for
$F_{1,\e}:= -\operatorname{div}(\varphi_{\e}(|\nabla \phi_{\e}|)\nabla\phi_{\e})$.
Define
\[
  B_{\e}^{ij} = g^{ij} + \frac{\varphi_{\e}^{\prime}}{\rho\varphi_{\e}}\phi_{\e}^{i}\phi_{\e}^{j}
\quad\quad\mbox{and}\quad\quad
\tilde{B}_{\e}^{ij} = g^{ij} - \frac{\phi_{\e}^i\phi_{\e}^j}{\rho^2},
\]
where $\rho = |\nabla \phi_{\e}|$.  Introducing the operator
\[
        \widehat{L}_{\e}(\Phi) \doteqdot
\operatorname{div}\left(\frac{p-2}{p-1}\phi_{\e}\varphi_{\e}
        B_{\e}\nabla\Phi\right),
\]
we compute
\begin{align}
\label{eq:F1approx}
\begin{split}
    \left(\widehat{\mathcal{L}}_{\e} -\pdt\right)F_{1,\e} &=
        2\zeta_{\e}\varphi_{\e}\left\langle \nabla F_{1, \e}, \nabla \phi_{\e}
        \right\rangle +
        \frac{p-2}{p-1}\zeta_{\e}F_{1,\e}^2\\
        &\quad
        +\frac{p}{p-1}\zeta_{\e}\varphi_{\e}^2
        \left[|\nabla\nabla\phi_{\e}|^2_{\tilde{B}_{\e}}
                + \operatorname{Rc}(\nabla \phi_{\e},\nabla\phi_{\e})\right]\\
&\quad
+ \frac{p}{p-1}\zeta_{\e}\varphi_{\e}^2\left(\frac{\rho\zeta^{\prime}_{\e}}{\zeta_{\e}}
        +2\zeta_{\e}\right)(S- H^2)\\
&\quad +\zeta_{\e}^2\varphi_{\e}^2\left(\frac{p}{p-1}\zeta_{\e} + 2\frac{\rho\zeta^{\prime}_\e}{\zeta_{\e}}\right)H^2
+ \frac{p-2}{p-1}\zeta^{\prime}_{\e}\varphi_{\e}^2H(\Delta \phi_{\e} - H),
\end{split}
\end{align}
where
\[
        S \doteqdot \frac{1}{\rho^2}g^{ik}
        \nabla_i\nabla_j\phi_{\e} \nabla_k\nabla_l \phi_{\e}
        \nabla^{j}\phi_{\e}\nabla^{l}\phi_{\e}
\quad\quad\mbox{and}\quad\quad H\doteqdot \frac{1}{\rho^4}
        \left[\nabla\nabla \phi_{\e}(\nabla\phi_{\e}, \nabla \phi_{\e})\right]^2.
\]
Now, $F_{1,\e} = -\varphi_{\e}B_{\e}^{ij}\nabla_{i}\nabla_{j}\phi_{\e}$,
and when $\rho \in [0, a_{\e}]\cup  [\e, \infty)$,
the final five terms on the left side of \eqref{eq:F1approx} reduce to
\begin{align}\notag
        &\frac{p-2}{p-1}\zeta_{\e}F_{1,\e}^2 + \frac{p}{p-1}\zeta_{\e}\varphi_{\e}^2
\left[|\nabla\nabla \phi_{\e}|_{B_{\e}}^2 + \operatorname{Rc}(\nabla\phi_{\e},
        \nabla\phi_{\e})\right]\phantom{EXTRASPACE}\\
\notag
        &\qquad\qquad\qquad\qquad\qquad\geq \frac{\zeta_{\e}}{p-1}\left(p-2 + \frac{p}{n}\right)F_{1,\e}^2
        \notag\\
        &\qquad\qquad\qquad\qquad\qquad\geq \frac{1}{\max\{1, p-1\}n\beta } F_{1, \e}^2.\label{eq:zetaconst}
\end{align}

When $\rho \in (a_{\e}, \e)$, we need to do a little extra estimation.
First, note that $S \geq H^2$ (so that the fourth term in \eqref{eq:F1approx}
is non-negative for sufficiently small $\epsilon$ by the assumptions
on $\zeta_{\e}^{\prime}$).  Also, one has
\begin{align}
\notag F_{1,\e}^2 &\leq \frac{n}{n-1}
\varphi_{\e}^2\left(\Delta \phi_{\e} - H\right)^2 +n \zeta_{\e}^2\varphi_{\e}^2 H^2\\
\label{eq:btildeest}
        &\leq n\varphi_{\e}^2|\nabla\nabla\phi_{\e}|^2_{\tilde{B}_{\e}}
        +n \zeta_{\e}^2\varphi_{\e}^2 H^2,
\end{align}
and similarly
for any $\mu > 0$, that
\begin{equation}\label{eq:xtermest}
        H\left(\Delta\phi_{\e} - H\right) \geq
                -\frac{\mu}{2}H^2
                -\frac{n-1}{2\mu}|\nabla\nabla \phi_{\e}|^2_{\tilde{B}_{\e}}.
\end{equation}
With these inequalities, equation \eqref{eq:F1approx}, and the assumptions on
$\zeta^{\prime}_{\e}$, we can find  a constant $C$, independent
of $\epsilon$, such that
\begin{align}
\notag
 &\left(\widehat{\mathcal{L}}_{\e} - 2\zeta_{\e}\varphi_{\e}\nabla \phi_{\e}
\cdot \nabla -\pdt\right)
F_{1,\e}\phantom{SPACESPACESPACESPACE}\\
&\qquad\qquad\qquad\qquad \geq
        \frac{\zeta_{\e}}{p-1}\left[ (p-2) F_{1,\e}^2
        +(p- C\e)\varphi_{\e}^2
\left(|\nabla\nabla\phi_{\e}|^2_{\tilde{B}_{\e}} +\zeta_{\e}^2
        H^2\right)\right] \notag\\
\label{eq:midcase}
&\qquad\qquad\qquad\qquad\geq \frac{1}{\max\{1, p-1\}}\left(p -2 + \frac{p - C\e}{n}\right)
F_{1,\e}^2.
\end{align}
Since we have already noted that the stronger inequality \eqref{eq:zetaconst}
holds when $\rho\leq a_{\e}$ or $\rho \geq \e$, one may appeal to the maximum
principle to deduce in general that
\[
    F_{1,\e} \leq \frac{n\bar{\beta}}{(1 - C\e)t}.
\]
Letting $\e \to 0$, and making use of appropriate energy estimates, as in \cite{EV2}, we can obtain a limiting $\phi$, and consequently a solution $v$ to
\eqref{eq:p-par}, for which the inequality \eqref{eq:par-global} holds.
\end{proof}

\section{Another nonlinear parabolic equation}

Note that with the above parabolic estimate, one can not recover the elliptic
result proved in Section 2, thus in this section, we consider instead the
parabolic equation associated to the operator
$\Lambda_p$. Let $u$ be a solution to the equation
\begin{equation}\label{pv2-equ}
\frac{\partial u}{\partial t} -\Lambda_p(u)=\frac{\partial
u}{\partial t} -\operatorname{div}(|\nabla u|^{p-2}\nabla u)+|\nabla
u|^p=0.
\end{equation}
This nonlinear parabolic equation has also been studied in the
literature. See, for example, \cite{V1} and the references therein.

The corresponding equation for $v=\exp(-\frac{u}{p-1})$ is
\begin{equation}\label{pv2-eqv}
\frac{\partial v^{p-1}}{\partial t} =(p-1)^{p-1}\operatorname{div}
\left(|\nabla v|^{p-2}\nabla v\right).
\end{equation}

Recall from Section 2 the operator $\mathcal{L}$ defined as
$$
\mathcal{L}(\psi)=\operatorname{div}(f^{p/2-1}A(\nabla
\psi))-pf^{p/2-1}\langle \nabla u, \nabla \psi\rangle,
$$
which is the linearized operator of $\Lambda_p$. Note that
$$
A^{ij}=g^{ij}+(p-2)\frac{v^i v^j}{|\nabla
v|^2}=g^{ij}+(p-2)\frac{u^i u^j}{|\nabla u|^2}.
$$

For any $u$, letting $f = |\nabla u|^2$, the proof of Lemma \ref{bochner1} yields,
the nonlinear Bochner formula
\begin{equation}\label{bochner-real}
\mathcal{L} (f) =2f^{p/2-1}(u_{ij}^2+R_{ij}u_iu_j)+2\langle \nabla
u, \nabla (\Lambda_p(u))\rangle +(\frac{p}{2}-1)|\nabla f|^2
f^{p/2-2}.
\end{equation}
A bit of computation together with this formula yields the following lemma.
\begin{lemma} \label{bochner3} Let $f=|\nabla u|^2$. Then
\begin{eqnarray}\label{pv2-equt}
\left(\frac{\partial }{\partial t}-\mathcal{L}\right)u_t &=&0\\
\left(\frac{\partial }{\partial t}-\mathcal{L}\right) f&=&
-2f^{p/2-1}(u_{ij}^2+R_{ij}u_iu_j)-(\frac{p}{2}-1)f^{p/2-2}|\nabla
f|^2. \label{pv2-eqf}
\end{eqnarray}
\end{lemma}

An immediate consequence is the following result.
\begin{corollary}\label{cor:grad}  For any $\alpha>0$, let $F_\alpha=|\nabla u|^p+ \alpha u_t=f^{p/2}+\alpha u_t$.
Then
\begin{equation}\label{bochner-key}
\left(\frac{\partial }{\partial t}-\mathcal{L}\right) F_\alpha
=-pf^{p-2}\left(|\nabla \nabla u|^2_A +R_{ij} u_i u_j\right).
\end{equation}
\end{corollary}
Notice that (\ref{pv2-equ}) implies that
$$
F_1=f^{p/2-1}\tr_A (\nabla\nabla u).
$$
Hence
$$
f^{p-2}|\nabla \nabla u|^2_A\ge \frac{1}{n}F_1^2.
$$
In the case that $M$ is compact with nonnegative Ricci curvature, we
obtain at once the following global estimate.
\begin{theorem}\label{thm:ly-ddnonl} Let $M$ be a compact manifold  with
nonnegative Ricci curvature. Let $v$ be a positive solution to
(\ref{pv2-eqv}). Then for any $p> 1$,
\begin{equation}\label{lyp1}
(p-1)^p\frac{|\nabla v|^p}{v^p}-(p-1)\frac{v_t}{v}\le \frac {n}{pt}.
\end{equation}
\end{theorem}
The result also holds on noncompact manifolds assuming the the left
hand side is bounded. Note that when $p=2$, the above reduces to Li-Yau's
estimate. Hence the corollary provides another nonlinear analogue of
Li-Yau's estimate for the heat equation.

\begin{remark} It is easy to check that for $p>1$, the function
\begin{equation}\label{ffs}
v_0=\left(\pi^{-n/2}
({p^*}^{p-1}pt)^{-n/p}\frac{\Gamma(n/2+1)}{\Gamma(n/p^*+1)}
\exp\left(-\frac{|x-x_0|^{p^*}}{(t{p^*}^{p-1}p)^{\frac{1}{p-1}}}
\right)\right)^{\frac{1}{p-1}} \end{equation} is a fundamental
solution of (\ref{pv2-eqv}) on $\mathbb{R}^n$, and achieves equality
in the estimate
(\ref{lyp1}). This demonstrates the sharpness of the estimate. Here
$p^*=\frac{p}{p-1}$.
\end{remark}

\begin{remark} It seems that an estimate in the above sharp form is not known except in
dimension one \cite{EV1}. It should be useful in the study of the
regularity of weak solutions of (\ref{pv2-eqv}).
\end{remark}

When $p=1$ we have the following result.
\begin{theorem} \label{thm:ly-ddnl2}Let $M$ be a compact manifold  with
nonnegative Ricci curvature. Let $u$ be a smooth solution to
(\ref{pv2-equ}). Then
\begin{equation}\label{lyp2}
|\nabla u|+u_t\le \frac {n-1}{t}.
\end{equation}
\end{theorem}

Since the equation becomes degenerate when $\nabla u =0$, the above computation needs extra justification.
First we introduce the notion of a weak solution to as in \cite{Hein}.
For any  compact subset $K$, we define the functional
$$
J^p_u(w; K)\doteqdot \int_{K\times\{t\}}\frac{1}{p}|\nabla w|^p +w(|\nabla u|^p+u_t)\, d\mu
$$
for $w\in \mathcal{C}\doteqdot \{\varphi \in C^0(M \times [0, \infty))\,|\,\varphi(\cdot, t)\in C^{0,1}(M), \mbox{ and  for every } x\in M, \varphi(x, \cdot)\in C^{0,1}_{loc}(0, \infty)\}$ such that $\{w\ne u\}\subset \subset M$.
A function $u$ is called a {\it continuous weak solution } to \eqref{pv2-equ} if
$$
J^p_u(u; K)\le J^p_u(w; K)
$$
for any $w\in \mathcal{C}$.
Note that when $u$ is  smooth (locally), satisfying (\ref{pv2-equ}) with $|\nabla u|\ne 0$, we have that
\begin{eqnarray*}
J^p_u(u; K)-J^p_u(w; K)&=&\int_K \frac{1}{p}(|\nabla u|^p-|\nabla w|^p)+(u-w)\operatorname{div}(|\nabla u|^{p-2}\nabla u)\, d\mu\\
&\le &\int_K \frac{1}{p}(|\nabla u|^p-|\nabla w|^p)-|\nabla u|^p+|\nabla u|^{p-1}|\nabla w|\, d\mu\\
&\le& 0
\end{eqnarray*}
by Young's inequality.  To obtain weak solutions more generally, we use an
$\epsilon$-regularization process, replacing the equation (\ref{pv2-equ}) by
the  approximate version
\begin{equation}
\label{eq:pv2approx}
    \frac{\partial u_{\epsilon}}{\partial t} = \Lambda_{p, \e}(u_{\e})\doteqdot
    \operatorname{div}(f_{\e}^{p/2 -1}\nabla u_{\e}) - f_{\e}^{p/2}
\end{equation}
where
$
 f_{\e} \doteqdot |\nabla u_{\e}|^2 + \e.
$
Since (\ref{eq:pv2approx}) is strictly parabolic, one can apply the
established theory of parabolic equations to obtain a solution $u_\epsilon$ and take the limit $\epsilon \to 0$  to obtain a {\it continuous weak solution} to (\ref{pv2-equ}) as in \cite{Hein}.
We next show that an estimate analogous to (in fact, somewhat stronger than)
\eqref{lyp2} can be obtained for solutions
$u_\epsilon$ to (\ref{eq:pv2approx}), and with this estimate recover \eqref{lyp2} in the limit
as $u_\epsilon \to u$.

It will be convenient to introduce the notation
\begin{equation*}
  \mathcal{L}_{\e}(\psi) \doteqdot D\Lambda_{p, \e}[u_{\e}](\psi)
  =\operatorname{div}(f^{p/2 -1}_{\e}A_{\e}\nabla\psi)
  - pf_{\e}^{p/2-1}\langle\nabla u_{\e}, \nabla \psi\rangle
\end{equation*}
for the linearization of the operator $\Lambda_{p, \e}$,
where
\[
 A_{\e} = \id + (p-2)\frac{\nabla u_{\e}\otimes\nabla u_{\e}}{f_{\e}}.
\]
It is easy to check that
\[
  \left(\mathcal{L}_{\e} - \pdt\right) \pd{u_{\e}}{t} = 0,
\]
and a computation as before establishes the following Bochner-type formula.
\begin{lemma}
\begin{eqnarray}\label{bochner-epsilon}
\mathcal{L}_{\e}(f_\e)&=&2f_\e^{p/2-1}\left(|\nabla \nabla u_{\epsilon}|^2+\Ric(\nabla u_\epsilon, \nabla u_\e)\right)+2\langle \nabla u_\e, \nabla (\Lambda_{p,\e}(u_\epsilon))\rangle \nonumber\\
&\,& +(\frac{p}{2}-1)|\nabla f_\e|^2f^{p/2-2}_\e.
\end{eqnarray}
In particular,
\begin{equation*}
  \left(\mathcal{L}_{\e} - \pdt\right) f^{p/2}_{\e} =
  p f_{\e}^{p-2}\left(\Ric(\nabla u_{\e}, \nabla u_{\e})
    + |\nabla\nabla u_{\e}|^2_{A_{\e}}\right).
\end{equation*}

\end{lemma}

Thus we obtain the analog of
Corollary \ref{cor:grad} for solutions to our approximate equation
\eqref{eq:pv2approx}.
\begin{lemma}\label{lem:approxgrad}
For all $\alpha > 0$,
\[
F_{\alpha, \epsilon} \doteqdot f_{\e}^{p/2} + \alpha \pd{u_{\e}}{t}
\]
satisfies
\begin{equation}\label{eq:approxgrad}
 \left(\mathcal{L}_{\e} - \pdt\right)F_{\alpha, \e} =p f_{\e}^{p-2}\left(\Ric(\nabla u_{\e}, \nabla u_{\e})
    + |\nabla\nabla u_{\e}|^2_{A_{\e}}\right).
\end{equation}
\end{lemma}

In the same way as before, we note
\[
   F_{1, \epsilon} = \operatorname{div}(f_{\e}^{p/2-1}\nabla u_{\e})
   = f_{\e}^{p/2 -1}\operatorname{tr}_{A_{\e}}(\nabla\nabla u_{\e})
\]
so
\[
   f_{\e}^{p-2}|\nabla\nabla u|^2_{A_{\e}} \ge \frac{1}{n}F_{1, \e}^2,
\]
which, with the above lemma, implies the following global estimate.
\begin{theorem}\label{thm:globalapproxge}
Suppose $u_{\e}$ is a positive solution to \eqref{eq:pv2approx} on the compact
manifold $M$ of non-negative Ricci curvature.  Then, for all $p > 1$ and
all $t > 0$,
\begin{equation}
  (|\nabla u_{\e}|^2 + \epsilon)^{p/2} + \pd{u_{\e}}{t} \leq \frac{n}{pt}.
\end{equation}
\end{theorem}
Hence Theorem \ref{thm:ly-ddnonl} and Theorem \ref{thm:ly-ddnl2} hold for the weak solutions obtained via the $\epsilon$-regularization process.

The Bochner-type formula (\ref{bochner-key}) is effective enough to
give a nonlinear entropy formula (see \cite{N1} for the entropy
formula for the heat equation). For this purpose, we first observe
that $\int_M v^{p-1}\, d\mu$ is preserved under the equation
(\ref{pv2-eqv}). A little less obvious, perhaps, is the following
conservation law.

\begin{proposition} For any smooth $\psi$,
\begin{equation}\label{pv2-hp1}
\frac{d}{dt} \int_M \psi v^{p-1}\, d\mu =\int_M
\left(\left(\frac{\partial}{\partial t}-\mathcal{L}\right)
\psi\right)v^{p-1}\, d\mu.
\end{equation}
In particular, if  $\psi$ satisfies
\begin{equation}\label{pv2-linear}
\left(\frac{\partial}{\partial t}-\mathcal{L}\right) \psi =0
\end{equation}
then $\int_M \psi v^{p-1}\, d\mu$ is a constant along the equation
(\ref{pv2-eqv}) and (\ref{pv2-linear}).
\end{proposition}
\begin{proof}
Notice that $\nabla v =-\frac{v}{p-1}\nabla u$, hence
$(p-1)^{p-1}|\nabla v|^{p-2}\nabla v=-v^{p-1}|\nabla u|^{p-2}\nabla
u$. Direct computation shows that
\begin{eqnarray*}
\frac{d}{dt}\int_M\psi v^{p-1}\, d\mu &=&\int_M \psi_t v^{p-1} +\psi
(v^{p-1})_t \\
&=&\int_M \left(\left(\frac{\partial}{\partial t}-\mathcal{L}\right)
\psi\right)v^{p-1}-\int_M |\nabla
u|^{p-2}\langle A(\nabla \psi), \nabla v^{p-1} \rangle\\
&\quad&-(p-1)\int_M |\nabla u|^{p-2}\langle \nabla u, \nabla \psi\rangle
v^{p-1}.
\end{eqnarray*}
The result follows from the observations that $\nabla
v^{p-1}=-v^{p-1}\nabla u$ and
\begin{eqnarray*}
-\int_M |\nabla
u|^{p-2}\langle A(\nabla \psi), \nabla v^{p-1} \rangle&=&
-\int_M |\nabla u|^{p-2} \langle \nabla \psi, \nabla v^{p-1}\rangle
\\&\quad &-(p-2) \int_M |\nabla u|^{p-4}\langle \nabla \psi, \nabla u\rangle
\langle \nabla u, \nabla v^{p-1}\rangle\\
&=& (p-1)\int_M |\nabla u|^{p-2}v^{p-1}\langle \nabla \psi, \nabla
u\rangle.
\end{eqnarray*}
\end{proof}

\begin{remark}\label{regular} In the above we do not really  need that $\psi$ is smooth.
The argument carries over assuming only the summability of the integrands.
\end{remark}

Let $(a_{ij})$ be the inverse of $(A^{ij})$. Explicitly,
$a_{ij}=g_{ij}-\frac{p-2}{p-1}\frac{u_i u_j}{f}$, and can be viewed as
a metric tensor. Let $v$ be a positive solution to
(\ref{pv2-eqv}) satisfying $\int_M v^{p-1}\, d\mu =1$. Define
$$
\mathcal{N}_p(v, t)=\int_M v^{p-1} u \, d\mu-\frac{n}{p}\log t
$$
and denote the first term on the right hand side by $N(v, t)$.
Note that $u=-\log(v^{p-1})$. Following \cite{N1}, (see also
\cite{FIN}), define $\mathcal{F}_p(v, t)=\frac{d}{ d t}
\mathcal{N}_p(v, t)$ and $\mathcal{W}_p(v, t)=\frac{d}{d t}
(t\mathcal{N}_p(v, t))$. More explicitly, motivated by (\ref{ffs}),
if we write
$$
v^{p-1}=\frac{1}{\pi^{n/2}({p^*}^{p-1}p)^{\frac{n}{p}}}\frac{\Gamma(n/2+1)}{\Gamma(n/p^*+1)}
\frac{e^{-\varphi}}{t^{\frac{n}{p}}},
$$
then define
$$
\mathcal{W}_p(\varphi, t)=\int_M (t|\nabla \varphi|^p+\varphi
-n)v^{p-1}\, d\mu.
$$
For the case $p=1$, it is helpful to write $\mathcal{W}$ in terms of $u$,
$$
\mathcal{W}_p(u, t)=\int_M \left(t|\nabla u|^p+u-\frac{n}{p}\log
t+\log
\left(\frac{1}{\pi^{n/2}({p^*}^{p-1}p)^{\frac{n}{p}}}\frac{\Gamma(n/2+1)}{\Gamma(n/p^*+1)}
\right)-n\right)e^{-u}\, d\mu,
$$
which  becomes
$$
\mathcal{W}_1(u, t)=\int_M \left( t|\nabla u|+u-n\log t +\log
\frac{\Gamma(n/2+1)}{\pi^{n/2}}-n\right)e^{-u}\, d\mu
$$
when $p=1$.
Note that when $p=2$, this is precisely the entropy
defined in \cite{N1}.

For the entropy quantity $\mathcal{W}_p(v,t)$
we have the following nonlinear entropy formula.
\begin{theorem}\label{pv2-entropy}
Let $v$ be a positive solution to (\ref{pv2-eqv}) satisfying $\int_M
v^{p-1}\, d\mu =1$. Then
\begin{equation}\label{pv2-ef1}
\frac{d}{d t}\mathcal{W}_p(v, t)=-tp\int_M
\left(\left|f^{p/2-1}\nabla_i \nabla_j u
-\frac{1}{tp}a_{ij}\right|^2_{A}+f^{p-2}R_{ij} u_i u_j\right)
v^{p-1}\, d\mu.
\end{equation}
\end{theorem}
When $p=2$, this recovers the entropy formula of
\cite{N1}.

For the proof we need the following result.
\begin{proposition}\label{pv2-energy}
\begin{eqnarray}\label{nash-mono}
\frac{d}{ dt} N(v, t)&=& \int_M F_1 v^{p-1}\, d\mu =\int_M
f^{p/2-1}\tr_A(\nabla \nabla u) v^{p-1}\,
d\mu,\\
\frac{d}{ dt} \int_M F_1 v^{p-1}\, d\mu &=& -p\int_M
f^{p-2}\left(|\nabla \nabla u|^2_A +R_{ij} u_i u_j\right) v^{p-1}\,
d\mu. \label{energy-mono}
\end{eqnarray}
\end{proposition}
\begin{proof} (of Proposition.) Direct calculation shows that
$$
\left(\frac{\partial}{\partial t} -\mathcal{L}\right) u =(p-1)
|\nabla u|^p -(p-2)\operatorname{div}(|\nabla u|^{p-2} \nabla u).
$$
Hence the first identity of Proposition \ref{pv2-energy} follows
from (\ref{pv2-hp1}) together with the observation
$$
\int_M F_1 v^{p-1}= \int_M \operatorname{div} (|\nabla
u|^{p-2}\nabla u) v^{p-1}=\int_M |\nabla u|^p v^{p-1}.
$$

The second identity of Proposition \ref{pv2-energy} follows from
(\ref{bochner-key}) and (\ref{pv2-hp1}).
\end{proof}
Note that in the case that the Ricci curvature of $M$ is nonnegative,
(\ref{energy-mono}) yields the monotonicity of `energy'
$$
\overline{\mathcal{F}}(v, t)\doteqdot \int_M |\nabla u|^p v^{p-1}.
$$
Now Theorem \ref{pv2-entropy} follows from the above proposition
similarly as in  \cite{N1} (also \cite{FIN}), by observing
that $ \frac{d \mathcal{W}_p}{ dt}=t\frac{d \mathcal{F}_p}{ dt}
+2\mathcal{F}_p$ and completing the square. This point of view is
taken from physics \cite{Ev, P}. The case $p=1$ (Theorem
\ref{efp=1}) can be shown similarly.

By ODE considerations and the Cauchy-Schwarz inequality, we also have
the following result.

\begin{corollary} Assume that $M$ has nonnegative Ricci curvature.
Then
\begin{equation}
\frac{d \mathcal{N}_p}{ dt}
=\mathcal{F}_p=\overline{\mathcal{F}}-\frac{n}{pt}\le 0.
\end{equation}
In particular, any positive ancient solution to (\ref{pv2-eqv}) must
be a constant.
\end{corollary}
 Summarizing we have that {\it if $M$ has nonnegative Ricci curvature then
 $\mathcal{N}_p(v, t)$ is a monotone non-increasing
 concave function in $\log t$.}

To deal with the potential vanishing of $\nabla u$, we work via the approximation scheme as before.
First, (\ref{nash-mono}) can be justified in view of Remark \ref{regular}. For (\ref{energy-mono}),
 we need to appeal to an approximation argument. Let $u_\e$ be a solution to (\ref{eq:pv2approx}) and
define $\varphi_\e$ similarly as before. We have the following pointwise computations.

\begin{proposition}\label{ptwise}
Let $u_\e$, $\varphi_\e$ and $f_\e$  be functions defined as before. Let
$$
V_\e=2\operatorname{div}(|\nabla u_\e|^{p-2}\nabla u_\e)-|\nabla u_\e|^p, \quad
W_\e=t(2\operatorname{div}(|\nabla \varphi_\e|^{p-2}\nabla
\varphi_\e)-|\nabla \varphi_\e|^p)+\varphi_\e -n.
$$
Then \begin{eqnarray} \ \ \ \left(\frac{\partial}{\partial
t}-\mathcal{L_\e}\right)V_\e &=& -pf_\e^{p-2}\left(|\nabla \nabla u_\e|^2_A
+\Ric(\nabla u_\e, \nabla u_\e)\right)\\
\left(\frac{\partial}{\partial t}-\mathcal{L}\right)W_\e
&=&-tp\left(\left|f_\e^{p/2-1}\nabla_i \nabla_j
u_\e-\frac{1}{tp}(a_\e)_{ij}\right|^2_{A_\e}
+\Ric(\nabla u_\e, \nabla u_\e)\right)\nonumber\\
&\quad&+(p-2)\left(|\nabla u_\e|^p-\operatorname{div}(|\nabla
u_\e|^{p-2}\nabla u_\e)\right).
\end{eqnarray}
\end{proposition}

Notice that
$$
\int_M V_\e v^{p-1}=2\int_M \langle \nabla u_\e, \nabla u\rangle |\nabla u_\e|^{p-1}v^{p-1} -|\nabla u_\e|^p v^{p-1}\to \int_M |\nabla u|^p v^{p-1}=\overline{\mathcal{F}}
$$
and
\begin{eqnarray*}
\frac{d}{dt}\int_M V_\e v^{p-1}&=& -p\int_M f_\e^{p-2
}\left(|\nabla \nabla u_\e|^2_{A_\e}+\Ric(\nabla u_\e, \nabla u_\e)\right)v^{p-1}\\
&\,& -\int_M f_\e^{p/2-1}\left(\langle \nabla V_\e, \nabla v^{p-1}\rangle +(p-2)\frac{\langle \nabla u_\e, \nabla V_\e\rangle}{f_\e}\langle \nabla u_\e, \nabla v^{p-1}\rangle\right)\\
&\, & -p\int_M f^{p/2-1}_\e \langle \nabla u_\e, \nabla V_\e\rangle v^{p-1}+\int_M f^{p/2-1}v^{p-1}\langle \nabla u, \nabla V_\e\rangle\\
&\to& -p\int_M f^{p-2}\left(|\nabla \nabla u|^2_{A}+\Ric(\nabla u,
\nabla u)\right)v^{p-1}
\end{eqnarray*}
as $\epsilon \to 0$. Hence (\ref{energy-mono}) can be established
for weak solutions which can be properly approximated by the
regularization process.

Strictly speaking, the above proof can be applied only for the case
when $M$ is compact. When $M$ is non-compact, further work is needed
to justify the integrations by parts involved. On complete noncompact
manifolds with nonnegative Ricci curvature (or for which the Ricci curvature
is bounded from below) the entropy formula has been established
rigorously in the forthcoming book \cite{chowetc1} for the case
$p=2$, with full justification for such manipulations. The argument there
can be adapted to the $p>1$ case. Since this is rather technical we
refer the readers to \cite{chowetc1} for details.

In \cite{N1}, the entropy formula for the heat equation was used to
show that  a manifold with nonnegative Ricci curvature and sharp
logarithmic Sobolev inequality must be isometric to $\R^n$. It is
interesting to study the relation between the $L^p$-logarithmic
Sobolev inequality and the entropy formula (\ref{pv2-ef1}). First
recall the sharp $L^p$-logarithmic Sobolev inequality (see for
example \cite{DpDG}).

\begin{theorem}\label{lp-sob} Let $M$ be Euclidean space $\R^n$.
Then for any $p>1$, and any $w$ with $\int |w|^p dx =1$,
\begin{equation}\label{log-sob1}
\int |w|^p \log |w|^p\, dx \le \frac{n}{p}\log \left(\mathcal{C}_{p,
n}\int |\nabla w|^p\, dx\right)
\end{equation}
where
$$
\mathcal{C}_{p,
n}=\frac{p}{n}\left(\frac{p-1}{e}\right)^{p-1}\pi^{-p/2}\left(\frac{\Gamma(n/2+1)}
{\Gamma(n/p^*+1)}\right)^{p/n}.
$$
\end{theorem}

The above inequality with sharp constant was established after the
work of \cite{DD, G} (see also \cite{E}), and has the following
form connected with the entropy quantity. (For the  $L^1$-version,
see \cite{Bo} as well as \cite{Be}.)

\begin{proposition}\label{lsobolev2}
If on $(M, g)$, (\ref{log-sob1}) holds then
$$
\mathcal{W}_p(\varphi, t)\ge 0
$$
for
$$v^{p-1}=\frac{1}{\pi^{n/2}({p^*}^{p-1}p)^{\frac{n}{p}}}\frac{\Gamma(n/2+1)}{\Gamma(n/p^*+1)}
\frac{e^{-\varphi}}{t^{\frac{n}{p}}}, $$ with $\int_M v^{p-1}=1$,
where $t$ is just a scaling factor.
\end{proposition}

\begin{proof} The proof is a technically more complicated version of the  $p=2$ case shown in
\cite{Chowetc}, pages 247--249. We leave the detailed checking to
the interested reader.
\end{proof}

We shall next prove the following characterization of $\R^n$
among manifolds with nonnegative Ricci curvature by the sharp
$L^p$-logarithmic Sobolev inequality.

\begin{theorem} Let $(M, g)$ be a complete, connected Riemannian manifold with
nonnegative Ricci curvature. Assume that (\ref{log-sob1}) holds for
some $p>1$ with the sharp constant $\mathcal{C}_{p,n}$ on $M$. Then
$M$ is isometric to $\R^n$.
\end{theorem}
\begin{proof} By an argument along the lines of \cite{P}, Section 3,  one can show
that for the fundamental solution $H$ to (\ref{pv2-eqv}),
$\lim_{t\to 0} \mathcal{W}_p(H, t)=0$. Then, the entropy formula ensures
that $\mathcal{W}_p(H, t)\le 0$. However, in the presence of a sharp
$L^p$-logarithmic Sobolev
inequality,  one may prove, as in Corollary \ref{lsobolev2}, that
$\mathcal{W}_p(H, t)\ge 0$ for $t>0$. Hence we can
conclude that $\frac{d}{dt} \mathcal{W}_p (H, t)\equiv 0$. Writing
$u=-(p-1)\log H$, and $f=|\nabla u|^2$ we have that
\begin{equation}\label{rid1}
u_{ij} =\frac{1}{tp f^{p/2-1}}\left(g_{ij}-\frac{p-2}{p-1}\frac{u_i
u_j}{|\nabla u|^2}\right)
\end{equation}
off of the set of critical points of $u$, which we shall denote by $\Sigma$.

Equation \eqref{rid1} implies in particular that
$\nabla\nabla u > 0$ on $M\setminus \Sigma$, so $\nabla\nabla u \geq 0$ on all of $M$.
It follows, then, that $p\in \Sigma$ if and only if $u(p) = \inf_{q\in M}u(q)\doteqdot m$.
We shall prove later that $\Sigma$ consists of
precisely one point; for the present we note that these facts imply
at least that $\Sigma$ is a connected (in fact, convex) subset of $M$.
Indeed, if $p_1$, $p_2\in \Sigma$, and $\gamma:[0,l]\to M^n$ is a unit-speed geodesic with
$\gamma(0) = p_1$ and $\gamma(l) = p_2$, then we have
\[
        \frac{d^2}{ds^2}u(\gamma(s)) = \left(\nabla_{\dot{\gamma}}
        \nabla_{\dot{\gamma}}u\right)(\gamma(s)).
\]
Since $u(\gamma(s))\geq u(\gamma(0))=u(\gamma(l)) = m$, the weak convexity of $u$ implies $u(\gamma(s)) \equiv m$ and so
$\gamma([0, l])\subset\Sigma$.

Let us now consider the structure of an arbitrary connected component $M_1$
of $M\setminus\Sigma$.  We may write $M_1\approx I \times N$ for an interval $I$
and a hypersurface $N$ defined by
\[
    N = M_1\cap \{u = u(p_0)\}
\]
for some arbitrary $p_0\in M_1$.
Now we choose
a orthonormal frame such that $e_1=\frac{\nabla u}{|\nabla u|}, e_2,
\cdots, e_n$. Then the equality (\ref{rid1}) implies that
\begin{eqnarray*}
u_{11}&=&\frac{1}{tp(p-1)u_1^{p-2}},\\
u_{1\beta}&=&0,\\
u_{\alpha \alpha}&=&\frac{1}{tp u_1^{p-2}}
\end{eqnarray*}
for $2\le \alpha, \beta\le n$. This shows that $|\nabla u|$ is
constant along the level set of $u$. We shall use the above information to
find the explicit form of the metric. Recall that $u_{\alpha
\beta} =h_{\alpha \beta} u_1$, where $h_{\alpha \beta}$ is the
second fundamental form of the level set hypersurface of $u$. Let $U=p|\nabla u|^{p-1}$. It is easy to see that the mean curvature of
this hypersurface is $H=\frac{n-1}{U}.$ Hence
$$
U_1 =1, \quad \Delta U= U_{11} +H U_1=\frac{n-1}{U}.
$$
If we fix a level set and parametrize the other ones using $s$, the
oriented  distance from this fixed one, we have that
$$
u_{ss}=\frac{1}{p(p-1) u_s^{p-2}}
$$
which implies that $(u_s^{p-1})_s=\frac{1}{p}$. Thus
\begin{equation}\label{eq:uder1}
        u_1=\left(\frac{s}{p}+B\right)^{\frac{1}{p-1}},
\end{equation}
where $B$ is the
value of $|\nabla u|^{p-1}$ on the originally fixed hypersurface.
This implies that the level hypersurfaces at distance $s$ have the
second fundamental form
$$
h_{\alpha \beta} =\frac{1}{s+pA} g_{\alpha \beta}.
$$
This further implies that the metric on $N \times I$ is of the form
\begin{equation}\label{eq:wp}
g_{ij}=ds^2 +(s+pB)^2 g_{N}
\end{equation}
where $g_N$ is the metric of the hypersurface $N$.
Since $M$ is complete, and $M_1$ a connected component of $M\setminus\Sigma$,
we see that $I = (-pB, \infty)$.

Now, from equation \eqref{eq:wp}, one may compute that
$K_{M}(X, \nabla u,  X, \nabla u) = 0$ along
$N$ for any $X$ tangential to $N$ and hence (either again from \eqref{eq:wp}
or from the Gauss equations)
that
\begin{equation}\label{eq:rc}
   \operatorname{Rc}_N(X, X) = \operatorname{Rc}_{M}(X, X) + (n-1)g_N(X, X).
\end{equation}
Since the Ricci curvature of $M$ is assumed to be non-negative,
it follows immediately from Myer's theorem that $N$ is compact and, moreover, that
the diameter of the hypersurface $N_s$ of signed distance $s\in I$ from $N$
satisfies the bound
\begin{equation}\label{eq:diam}
        \operatorname{diam}(N_s)\leq \pi (s + pB).
\end{equation}
Thus $\lim_{s\to -pB} \operatorname{diam}(N_s) = 0$ and $\partial M_1$ must consist of a single point $\mathcal{O}\in \Sigma$.  In fact, we must have $M = M_1 \cup \{\mathcal{O}\}$ as
otherwise the removal of $\mathcal{O}$ would disconnect $M$ -- an impossibility
if $\operatorname{dim}(M) \geq 2$.  Thus $M = M_1 \cup \{\mathcal{O}\}$ in this case.  (Note that since
the foregoing considerations imply that $M$ must be non-compact, the theorem has already been established in the case $\operatorname{dim}(M) = 1$.) So $N$ (and hence all $N_s$) are
geodesic spheres about the point $\mathcal{O}$, and the representation \eqref{eq:wp}
on $M\setminus \{\mathcal{O}\}$ implies that $M$ is isometric to $\mathbb{R}^n$.
\end{proof}

\begin{remark} Here we assumed the existence and regularity of the fundamental solution, together with the finiteness of all the integrals involved in the entropy. These assumptions certainly require justification.
\end{remark}

\section{Localization}
In this section, carrying on in the notation of the last,
 we localize the estimate (\ref{lyp1}) of Theorem
\ref{thm:ly-ddnonl}. Here, in considering
the parabolic equation (\ref{pv2-equ}) we encounter some technical
complications beyond those encountered in Section 2 for the derivation of the local estimate
for the elliptic version, including additional complications caused by the
nonlinearity/degeneracy of the equation.

First we focus on the case $1< p< 2$. As in Section 2, assume that
$u$ is a solution on $B(x_0, R)$ and let $\theta(t): [0, \infty)\to
[0,1]$ be a cut-off function satisfying
$\frac{(\theta')^2}{\theta}\le 10$ and $\theta''\ge -10\theta\ge
-10$. Now let $\eta(x)=\theta^p\left(\frac{r(x)}{R}\right)$. It is
easy to see that
\begin{equation}\label{lc1}
\frac{|\nabla \eta|}{\eta^{1-\frac{1}{p}}}\le \frac{20}{R}.
\end{equation}
The proof of Lemma \ref{cf-est1} implies the following estimate.

\begin{lemma}\label{cf-est2} Assume that the sectional curvature is
bounded from below on $B(x_0, R)$ by $-K^2$. Then
\begin{equation}\label{cut-off-ieq1}
\Delta \eta +(p-2)\frac{\eta_{ij}u_i u_j}{f}\ge -\frac{C_1(n,
KR)}{R^2}
\end{equation}
where $C_1(n, KR)=40(n+p-2)(1+KR)-20$.
\end{lemma}

\begin{theorem} Assume that the sectional curvature of $M$ on $B(x_0, R)$ satisfies $K_M \ge -K^2$ and
$u$ is a solution of (\ref{pv2-equ}). Assume further that $1<p<2$
and $u_t\le0$. Then for any $\alpha
>1$, on $B(x_0, \frac{R}{2})$,
\begin{equation}\label{loc-estfin}
|\nabla u|^p+\alpha u_t\le \max\left\{\left( 8n\alpha^2
\left(\frac{C_2'}{R^2}+\frac{C_3}{t^{\frac{2}{p}}}\right)\right)^{\frac{p}{2}},\quad
\left(4n(n-1)K^2\right)^{p/2} \left(\frac{\alpha}{\alpha
-1}\right)^p\right\}.
\end{equation}
\end{theorem}

\begin{remark} The above result is a parabolic analogue of Theorem
\ref{main1}. Indeed, if $u$ is a solution to the elliptic problem,
we have $u_t=0$, hence one can apply the theorem and recover
the elliptic estimate (albeit with a worse constant).
\end{remark}

\begin{proof}
Clearly, $tF_\alpha \eta$ achieves a maximum over $B(x_0,
R)\times[0, T]$ somewhere. For the purpose of bounding $F_\alpha$
from above, we may assume that $\eta F_\alpha>0$ at this maximum
point $(y_0, t_0)$; this implies in particular that $t_0>0$.
Applying the maximum principle at
$(y_0, t_0)$,  we have that $ \left(\frac{\partial}{\partial
t}-\mathcal{L}\right)(tF_\alpha \eta)\ge 0 $.

 Direct computation shows that
\begin{eqnarray*}
\left(\frac{\partial}{\partial t}-\mathcal{L}\right)(tF_\alpha \eta)
&=& t\eta\left(\frac{\partial}{\partial t}-\mathcal{L}\right)
F_\alpha +tF_\alpha \left(\frac{\partial}{\partial
t}-\mathcal{L}\right) \eta -2tf^{p/2-1}\langle \nabla \eta, \nabla
F_\alpha\rangle_A\\
&\quad & +\eta F_\alpha.
\end{eqnarray*}
Here $\langle \cdot, \cdot \rangle_A$ denotes the product with
respect to $A^{ij}$. Using that $\nabla (\eta F_\alpha)=0$
at the maximum point, with formula (\ref{bochner-key}) and the simple
estimate on $\frac{|\nabla \eta|_A^2}{\eta}$ from Section 2 we find
that
\begin{eqnarray*}
0&\le& \left(\frac{\partial}{\partial
t}-\mathcal{L}\right)(tF_\alpha \eta)\\
 &\le& -pt\eta
f^{p-2}\left(|\nabla \nabla u|^2_A -(n-1)K^2 f\right) +tF_\alpha
\left(\frac{\partial}{\partial
t}-\mathcal{L}\right) \eta \\
&\quad& +2tf^{p/2-1} F_\alpha \frac{|\nabla \eta|^2}{\eta} +\eta
F_\alpha.
\end{eqnarray*}
We first derive an effective estimate on
$F_\alpha\left(\frac{\partial}{\partial t}-\mathcal{L}\right) \eta$.
Direct calculation using (\ref{pv2-equ}) yields
\begin{eqnarray*}
F_\alpha\left(\frac{\partial}{\partial t}-\mathcal{L}\right)
\eta&=&(2-p)F_{\alpha}f^{p/2-2}\left(\langle \nabla \eta, \nabla f
\rangle -\frac{\langle \nabla u, \nabla \eta\rangle \langle \nabla
u, \nabla
f \rangle }{f}\right)\\
&\quad& -\left(\Delta \eta +(p-2)\frac{\eta_{ij}u_i
u_j}{f}\right)f^{p/2-1}F_\alpha+pf^{p/2-1}F_\alpha\langle \nabla u,
\nabla
\eta\rangle\\
&\quad& +(2-p)F_1 F_\alpha \frac{\langle \nabla u, \nabla
\eta\rangle }{f}.
\end{eqnarray*}
We shall estimate each term in turn. Define $\tilde
A^{ij}=g^{ij}-\frac{u^i u^j}{f}$. It is easy to see that
\begin{eqnarray*}
I&=&(2-p)F_{\alpha}f^{p/2-2}\left(\langle \nabla \eta, \nabla f
\rangle -\frac{\langle \nabla u, \nabla \eta\rangle \langle \nabla
u, \nabla f \rangle }{f}\right)\\
&\le& (2-p)F_\alpha f^{p/2-2}\langle \nabla \eta, \nabla f\rangle_{\tilde A}\\
&\le & (2-p)F_\alpha f^{p/2-2}|\nabla \eta| |\nabla f|_{\tilde A}.
\end{eqnarray*}
Choose a normal frame so that $e_1=\frac{u_1}{|\nabla u|}$. Then
$u_j=0$ for all $j\ge 2$, and $f_k=2u_{1k} u_1$ for $k\ge 1$. Hence
$$
F_\alpha f^{p/2-2}|\nabla f|_{\tilde A} =2F_\alpha
f^{p/2-2}\left(\sum_{j\ge 2} u_{j1}^2 u_1^2\right)^{1/2}=2F_\alpha
f^{p/2-3/2}\left(\sum_{j\ge 2} u_{j1}^2 \right)^{1/2}.
$$
On the other hand,
\begin{eqnarray*}
|\nabla \nabla u|^2_{A}&=& u^2_{ik}+2(p-2) u_{k1}^2 +(p-2)^2
u_{11}^2\\
&=& (p-1)^2 u_{11}^2 +2(p-1)\sum_{j\ge 2} u_{j1}^2+\sum_{j,k\ge
2}u_{jk}^2\\
&\ge& 2(p-1)\sum_{j\ge 2} u_{j1}^2 +\frac{1}{n}\left(\tr_A (\nabla
\nabla u)\right)^2.
\end{eqnarray*}
Hence we have that
\begin{equation}\label{detail1}
I\le \frac{2-p}{2}\eta f^{p-2}|\nabla \nabla u|^2_{A}
+\frac{2-p}{p-1}\frac{|\nabla \eta|^2}{\eta}\frac{F_\alpha
^2}{f}-\frac{2-p}{2n}F_1^2 f^{p-2}\eta.
\end{equation}
Lemma \ref{cf-est2} implies that
\begin{equation}\label{detail2}
II=-\left(\Delta \eta +(p-2)\frac{\eta_{ij}u_i
u_j}{f}\right)f^{p/2-1}F_\alpha\le \frac{C_1(n, R,
k)}{R^2}f^{p/2-1}F_\alpha.
\end{equation}
Furthermore, it is easy to see that
\begin{equation}\label{detail3}
III=pf^{p/2-1}F_\alpha\langle \nabla u, \nabla \eta\rangle\le
pf^{\frac{p-1}{2}}F_\alpha |\nabla \eta|
\end{equation}
and
\begin{equation}\label{detail4}
IV=(2-p)F_1 F_\alpha \frac{\langle \nabla u, \nabla \eta\rangle
}{f}\le (2-p)F_1F_\alpha |\nabla \eta|f^{-1/2}.
\end{equation}

Since  $f^{p-2}|\nabla \nabla u|_A^2 \ge \frac{1}{n} F^2_1$ we split
$\eta F_\alpha = y-\alpha z$ with $y=\eta |\nabla u|^p$ and $z=\eta
u_t$. Also we can write $\eta^2 F_1^2$ as
\begin{equation}
(y-z)^2=\frac{1}{\alpha^2}(y-\alpha z)^2+2\frac{\alpha
-1}{\alpha^2}(y-\alpha z) y +\left(\frac{\alpha-1}{\alpha}\right)^2
y^2.
\end{equation}
Note $y-z=\frac{1}{\alpha}(y-\alpha z)+(1-\frac{1}{\alpha})y>0$.
Combining (\ref{detail1})--(\ref{detail4}), we arrive at
\begin{eqnarray*}
0&\le& -\frac{3p-2}{2}\eta^2 f^{p-2}|\nabla \nabla u|_A^2
+(n-1)p\eta^2
f^{p-1}K^2 +\frac{2-p}{p-1}|\nabla \eta|^2 F_\alpha^2 f^{-1}\\
&\, & +\frac{C_1(n, KR)}{R^2}\eta
f^{p/2-1}F_\alpha+pf^{\frac{p-1}{2}}F_\alpha |\nabla \eta|\eta+(2-p)
F_1 F_\alpha |\nabla \eta|\eta f^{-1/2}\\
&\,& +f^{p/2-1}F_\alpha |\nabla \eta|^2 +\frac{\eta^2 F_\alpha}{t}\\
 &\le & -\frac{3p-2}{2n} (y-z)^2+p(n-1)y^{2-\frac{2}{p}}K^2
+\frac{2-p}{p-1}(y-\alpha z)^2 y^{-\frac{2}{p}}\frac{|\nabla
\eta|^2}{\eta
^{2-\frac{2}{p}}}\\
&\, & +\frac{C_1(n, KR)}{R^2} (y-\alpha
z)y^{1-\frac{p}{2}}+p(y-\alpha z) y^{1-\frac{1}{p}}\frac{|\nabla
\eta|}{\eta^{1-\frac{1}{p}}}\\
&\,& +(2-p)(y-z)(y-\alpha z) y^{-\frac{1}{p}}\frac{|\nabla
\eta|}{\eta^{1-\frac{1}{p}}}+(y-\alpha z)
y^{1-\frac{2}{p}}\frac{|\nabla \eta|^2}{\eta^{2-\frac{2}{p}}}+\frac{y-\alpha z}{t}\\
&\le &  -\frac{1}{2n}\left(\frac{1}{\alpha^2}(y-\alpha
z)^2+2\frac{\alpha -1}{\alpha^2}(y-\alpha z) y
+\left(\frac{\alpha-1}{\alpha}\right)^2 y^2\right)\\
&\, &+2(n-1)y^{2-\frac{2}{p}}K^2
+\frac{50}{R^2}\frac{2-p}{p-1}(y-\alpha z)^2
y^{-\frac{2}{p}}+\frac{C_1(n, KR)}{R^2} (y-\alpha
z)y^{1-\frac{2}{p}}\\
&\,& +\frac{20}{R}(y-\alpha z) y^{1-\frac{1}{p}}+\frac{(2-p)}{\alpha
R}(y-\alpha z)^2
y^{-\frac{1}{p}}+\frac{(2-p)(\alpha-1)}{R\alpha}(y-\alpha z)
y^{1-\frac{1}{p}}\\
&\,&+\frac{100}{R^2}(y-\alpha z) y^{1-\frac{2}{p}}+\frac{y-\alpha
z}{t}.
\end{eqnarray*}
The above gives a local estimate if we assume  $u_t\le 0$ (namely
$z\ge 0$). (Assuming $u_t\le 0$ initially,
this can be established by the maximum principle, in
view of (\ref{pv2-equt}).)
In this case, we have that $y-\alpha z\le y$.

We make use of the above computation as follows. If we do not have
\begin{equation}\label{loc-est1}
y^{2/p}\le 4n(n-1)K^2\left(\frac{\alpha}{\alpha-1}\right)^2,
\end{equation}
which implies that $y-\alpha z$ is bounded by the same bound, we can
drop the terms only involving the powers of $y$ to still have the
last estimate on the functions of $y-\alpha z$.

 By the elementary inequality $2ab\le a^2+b^2$, we have that
 \begin{eqnarray*}
&\, &\quad\quad(y-\alpha z)\left(-\frac{\alpha-1}{n \alpha ^2} y
+\frac{20+(2-p)\frac{\alpha-1}{\alpha}}{R}y^{1-\frac{1}{p}}+\frac{C_1(n,
KR)+100}{R^2}y^{1-\frac{2}{p}}\right) \\
&\le&  (y-\alpha z)\left(
C_1+100+\left(20+(2-p)\frac{\alpha-1}{\alpha}\right)^2
\frac{n\alpha^2}{4(\alpha
-1)}\right)\frac{1}{R^2}y^{1-\frac{2}{p}}\\ &\le&
\frac{1}{R^2}\left(
C_1+100+\left(20+(2-p)\frac{\alpha-1}{\alpha}\right)^2
\frac{n\alpha^2}{4(\alpha -1)}\right)(y-\alpha z)^{2-\frac{2}{p}}.
 \end{eqnarray*}
 Here we have used $y-\alpha z\le y$. This takes care of the terms in which $(y-\alpha z)$
appears linearly.

Making use of $y-\alpha z\le y$ again, we have that
\begin{eqnarray*}
0&\le & -\frac{1}{2n \alpha^2}(y-\alpha
z)^2+\frac{C_2}{R^2}(y-\alpha z)^{2-\frac{2}{p}}+\frac{2-p}{\alpha
R}(y-\alpha z)^{2-\frac{1}{p}}+\frac{y-\alpha z}{t}
\end{eqnarray*}
with
$$
C_2=C_1(n, KR)+100+(20+(2-p)\frac{\alpha-1}{\alpha})^2
\frac{n\alpha^2}{4(\alpha -1)}+50\frac{2-p}{p-1}.
$$
Another use of the above elementary inequality reduces it to
\begin{equation}\label{fin1}
0\le-\frac{1}{4n \alpha^2}(y-\alpha z)^2+\frac{C_2'}{R^2} (y-\alpha
z)^{2-\frac{2}{p}}+\frac{y-\alpha z}{t}
\end{equation}
with $C_2'=C_2+n(2-p)^2$. Using Young's inequality $ab\le
\frac{a^{p'}}{p'}+\frac{b^{q'}}{q'}$ with
$\frac{1}{p'}+\frac{1}{q'}=1$,  we finally have that
\begin{equation}\label{fin2}
0\le-\frac{1}{8n \alpha^2}(y-\alpha
z)^2+\left(\frac{C_2'}{R^2}+\frac{C_3}{t^{\frac{2}{p}}}\right)(y-\alpha
z)^{2-\frac{2}{p}}
\end{equation}
where
$$
C_3=\frac{p}{2}\left(4n\alpha^2 (2-p)\right)^{\frac{2-p}{2}}.
$$
From (\ref{fin2}) it is then easy to derive the theorem.\end{proof}

\bibliographystyle{amsalpha}

\end{document}